\documentclass[11pt,reqno,a4paper]{amsart}
\usepackage{geometry,mathtools,diffcoeff,mathabx,mathrsfs,lmodern}
\numberwithin{equation}{section}

\usepackage[colorlinks,urlcolor=blue,citecolor=red,linkcolor=blue,linktocpage,bookmarksnumbered,bookmarksopen,pagebackref]{hyperref}

\usepackage[capitalize]{cleveref}

\newcommand{\R}{\mathbb R}
\newcommand{\N}{\mathbb N}
\newcommand{\X}{\mathcal{D}^{s,2}_0(\Omega)}
\newcommand{\epsi}{\varepsilon}
\newcommand{\slap}{(-\Delta)^s}

\def\XXint#1#2#3{{\setbox0=\hbox{$#1{#2#3}{\int}$} \vcenter{\vspace{-1pt}\hbox{$#2#3$}}\kern-.5\wd0}}
\def\Xint#1{\mathchoice {\XXint\displaystyle\textstyle{#1}}{\XXint\textstyle\scriptstyle{#1}}{\XXint\scriptstyle\scriptscriptstyle{#1}}{\XXint\scriptscriptstyle\scriptscriptstyle{#1}}\!\int}
\def\intmed{\Xint{-}}

\newtheorem{theorem}{Theorem}[section]
\newtheorem{proposition}[theorem]{Proposition}
\newtheorem{lemma}[theorem]{Lemma}
\newtheorem{corollary}[theorem]{Corollary}

\newtheorem{thmx}{Theorem}

\theoremstyle{definition}
\newtheorem{definition}[theorem]{Definition}
\newtheorem*{ack}{Acknowledgments}
\newtheorem*{plan}{Plan of the paper}

\theoremstyle{remark}
\newtheorem{remark}[theorem]{Remark}

\DeclareMathOperator{\dist}{dist}
\DeclareMathOperator{\Tail}{Tail}
\DeclareMathOperator*{\ess}{ess\,sup}

\DeclarePairedDelimiter\abs{\lvert}{\rvert}
\DeclarePairedDelimiter\norm{\lVert}{\rVert}
\DeclarePairedDelimiter\tonda{(}{)}
\DeclarePairedDelimiter\quadra{[}{]}
\DeclarePairedDelimiter\graffa{\{}{\}}
\renewcommand{\phi}{\varphi}

\title{A non-local semilinear eigenvalue problem}

\author[G. Franzina]{Giovanni Franzina}
\address[G. Franzina]{Istituto per le Applicazioni del Calcolo ``M. Picone''
\newline\indent
Consiglio Nazionale delle Ricerche
\newline\indent 
Via dei Taurini 19, 00185 Roma, Italy}
\email{giovanni.franzina@cnr.it}

\author[D. Licheri]{Danilo Licheri}
\address[D. Licheri]{Dipartimento di Matematica e Informatica
\newline\indent
Università degli Studi di Cagliari
\newline\indent
Via Ospedale 72, 09124 Cagliari, Italy}
\email{dani.licheri97@gmail.com}

\makeatletter
\@namedef{subjclassname@2020}{\textup{2020} Mathematics Subject Classification}
\makeatother


\begin{document}

\maketitle

\begin{abstract}
For a non-local semilinear eigenvalue problem, we prove simplicity and isolation of the first eigenvalue with homogeneous Dirichlet boundary conditions on open sets supporting a suitable compact Sobolev embedding.
\end{abstract}


\begin{center}
\begin{minipage}{10cm}
\small
\tableofcontents
\end{minipage}
\end{center}

\section{Introduction}
This paper concerns a semilinear eigenvalue
problem for the fractional Laplace operator with homogeneous Dirichlet
boundary conditions in $N$-dimensional Euclidean spaces with applications
to a model for non-local filtration in a porous medium. We recall that, given $s\in(0,1)$, the $s$-Laplacian of a smooth
function $u$ on $\R^N$ is defined, up to a normalisation constant depending only on $N$ and $s$, by the formula
\begin{equation}
\label{2}
(-\Delta)^su(x) = \lim_{\epsi\to0^+} \int_{\R^N\setminus B_\epsi(x)} \frac{u(x)-u(y)}{\abs{x-y}^{N+2s}}\,dy
\end{equation}

The right hand side is usually
multiplied by the quantity $4^s \Gamma\tonda*{\frac{N}{2}+s}/\tonda*{\pi^{N/2}|\Gamma(-s)|}$, which has a precise degenerate behaviour both as $s\to0^+$ and as $s\to1^-$. The specific normalisation
choice has no bearing for the matter of this paper and will be, therefore, omitted.

By classical spectral theory in Hilbert spaces, it is known that the eigenvalue problem
\[
	(-\Delta)^s u=\lambda u
\]
in a bounded open set $\Omega\subset\R^N$, 
with Dirichlet conditions $u=0$ in the complement $\R^N\setminus\Omega$, has non trivial solutions for a discrete set of real numbers $\lambda$, which either is empty
or consists of an unbounded non-decreasing sequence of \emph{eigenvalues}.
The corresponding \emph{eigenfunctions} are the stationary points of the double integral
\begin{equation}
\label{1}
	\int_{\R^N}\int_{\R^N} \frac{(u(x)-u(y))^2}{\abs{x-y}^{N+2s}}\,dx\,dy
\end{equation}
subject to an $L^2(\Omega)$-constraint. 

The variational problem under an $L^q(\Omega)$-constraint, with
$q\neq2$, leads one to a different
non-local semilinear elliptic boundary value problem, formally
 \begin{equation}\label{2qlambda}
\begin{cases*}		
\slap u=\lambda \norm u_{L^q(\Omega)}^{2-q}\abs u^{q-2}u & in $\Omega$\\
u=0 & in $\R^N\setminus\Omega$
\end{cases*}
\end{equation}
Any fixed solution $u$ of \eqref{2qlambda}, if multiplied by a specific constant depending on $u$, solves 
the \emph{fractional Lane-Emden equation} 
\begin{equation}
\label{LE}
	\slap u = \abs u^{q-2}u \qquad \text{in $\Omega$}
\end{equation}
with $u=0$ in $\R^N\setminus\Omega$. 

The largest lower bound for the collection $\mathfrak{S}(\Omega,s,q)$ of all positive numbers $\lambda$ for which
\eqref{2qlambda} admits a non-trivial solution is called the \emph{first $q$-semilinear $s$-eigenvalue}
\begin{equation}
\label{lambda1}
	\lambda_1(\Omega,s,q)=\inf_{\varphi\in C^\infty_0(\Omega)}\graffa*{
	\int_{\R^N}\int_{\R^N} \frac{(\varphi(x)-\varphi(y))^2}{\abs{x-y}^{N+2s}}\,dx\,dy :
	\int_\Omega{\abs\varphi^q\,dx}=1}
\end{equation}
In some cases, for example whenever $\Omega$ has finite $N$-dimensional volume,
the embedding $\X\hookrightarrow L^q(\Omega)$ is compact, which assures
the infimum to be achieved.

For $q\in(1,2)$, in fact,
a necessary and \emph{sufficient} condition that the embedding be compact is that it be continuous (see \cite[Theorem 1.3]{F19}). Hence,
we have the following existence and uniqueness result.

\begin{thmx}\label{teoa} 
Let $N\ge1$, $s\in(0,1)$, $q\in(1,2)$ and let $\Omega\subset\R^N$ be an open set with $\lambda_1(\Omega,s,q)>0$.
Up to a multiplicative constant,
there exists a unique eigenfunction achieving the minimum in \eqref{lambda1}.
The first eigenfunction has constant sign, and the first eigenvalue
is the unique one admitting eigenfunctions with this property.
\end{thmx}

We also prove a uniqueness result for $q>2$ smaller than a suitable threshold depending on $\Omega$ (see \cref{3.2}). 

The proof of \cref{teoa} follows standard methods (cf.~\cite{BF14}).
Its conclusion implies
the uniqueness of positive \emph{least energy solutions} of \eqref{LE}, i.e., 
positive solutions of the fractional Lane-Emden equation,
under homogeneous Dirichlet boundary conditions, that minimise the energy functional
\begin{equation}
\label{energy}
	\frac12	\int_{\R^N}\int_{\R^N} \frac{(\varphi(x)-\varphi(y))^2}{\abs{x-y}^{N+2s}}\,dx\,dy
	-\frac1q \int_\Omega{\abs\varphi^q\,dx}
\end{equation}
Thus, for $q\in(1,2)$, 
to every open set $\Omega$ with $\lambda_1(\Omega,s,q)>0$
we can associate the positive least energy solution $w_{\Omega,s,q}$, also called the \emph{fractional Lane-Emden density of $\Omega$} (in fact, the definition can be given for arbitrary open sets in $\R^N$, see Section~\ref{techn2}
for details).

Remarkably, in analogy with the local case (cf.~\cite{BFR}), a negative power of the fractional Lane-Emden density of $\Omega$ appears as a singular weight in a sort of Hardy inequality:
\begin{equation}
\label{LEineq-intro}
	\int_{\Omega} \frac{u^2}{w_{\Omega,s,q}^{2-q}}\,dx\le \int_{\R^{N}}\int_{\R^N} \frac{(u(x)-u(y))^2}{\abs{x-y}^{N+2s}}\,dx\,dy
	\qquad \text{for all $u\in C^\infty_0(\Omega)$}
\end{equation}
We refer to \cref{prop:LEineq} for more details about \eqref{LEineq-intro}. A better known Hardy-type inequality in the fractional setting would involve the distance to the boundary,
instead:
\begin{equation}
\label{hardy-intro}
	\int_{\Omega} \frac{u(x)^2}{
	\dist(x,\partial\Omega)^{2s}}\,dx\le \int_{\R^{N}}\int_{\R^N} \frac{(u(x)-u(y))^2}{\abs{x-y}^{N+2s}}\,dx\,dy
	\qquad \text{for all $u\in C^\infty_0(\Omega)$}
\end{equation}
Inequality \eqref{hardy-intro} always holds, e.g., on bounded Lipschitz sets (see Section~\ref{techn2}).

From inequalities \eqref{LEineq-intro} and \eqref{hardy-intro}, thanks to fractional Hopf's lemma, we can infer the local uniqueness in $L^1(\Omega)$ for positive solutions of fractional Lane-Emden equation \eqref{LE}; this means that the positive least energy solution $w_{\Omega,s,q}$ of \eqref{LE} is isolated in $\mathcal{D}_0^{s,2}(\Omega)$ with respect to the topology of the convergence in $L^1(\Omega)$. We refer to \cref{LEL1uniq} for a more precise statement. By a strategy borrowed from \cite{BDF}, where the result was first proved in the local case, we draw the following consequence.

\begin{thmx}\label{teo-isol}
Let $N\ge1$, $s\in(0,1)$, $q\in(1,2)$ and let $\Omega\subset\R^N$ a bounded open set with $C^{1,1}$ boundary. Then, $\lambda_1(\Omega,s,q)$ is isolated, i.e.,
there exist no sequence of $q$-semilinear $s$-eigenvalues converging to it.
\end{thmx}

Little more is known about higher eigenvalues, except that they form a closed set that does not accumulate to  $\lambda_1(\Omega,s,q)$. It is indeed possible to assemble 
an unbounded sequence of $q$-semilinear $s$-eigenvalues by means of
standard critical point theory (see \cref{rmk:higher} below) but
it is not known if that gives a complete description of the 
$q$-semilinear $s$-spectrum, nor is it known if the latter is
a discrete set.

Given $m>1$, simplicity (\cref{teoa}) and isolation (\cref{teo-isol}) of the first $q$- semilinear $s$-eigenvalue with
 $q=1+\frac1m$ have implications on the
long-time behaviour of solutions to 
the initial-boundary value problem for the the \emph{fractional porous media equation} (see \cite{VV})
\begin{equation*}
\begin{cases*}
	\partial_t v + \slap \tonda*{\abs v^{m-1}v}=0 & in $\Omega\times(0,T)$\\
	v=0 & in $\tonda*{\R^N\setminus\Omega}\times(0,T)$\\
	v=v_0 & in $\Omega\times\{0\}$
\end{cases*}
\end{equation*}
We hope to return to this topic in the future, while in this paper we limit our attention to the elliptic problem.

\begin{plan}
In Section~\ref{sec2}, after framing our problem in 
appropriate function spaces we introduce the fractional semilinear eigenvalue problem and the non-local Lane-Emden density. More details on the former are provided in Section~\ref{techn1}, and various properties of
the latter are discussed in Section~\ref{wfrac}.
The preliminary results are used to prove
\eqref{LEineq-intro} in Section~\ref{techn2}, where
\eqref{hardy-intro} is also proved. 
Then, Section~\ref{sec:isolation} is devoted to the isolation of positive solutions of the non-local Lane-Emden equation;
eventually, all the partial results are used in Section~\ref{sec:mainproof} to prove \cref{teoa} and \cref{teo-isol}.
\end{plan}

\begin{ack}
The authors are grateful to Lorenzo Brasco for fruitful discussions on the problem considered in the paper, in particular on the various regularity estimates and on related topics, leading the authors to improve the quality of this paper: he is acknowledged especially for pointing attention to the Hardy-type inequality \eqref{hardy-intro}.

D. Licheri is supported by the grant ``Non-homogeneous eigenvalue problems and applications'' (University of Cagliari, 2022).
\end{ack}

\section{Framework and (pseudo) differential equations}\label{sec2}
Throughout this paper, we fix an integer $N\ge1$, a real number $s\in(0,1)$ and an open set $\Omega\subset\R^N$. The square root of 
\begin{equation}
\label{pre-1}
\int_{\R^N}\int_{\R^N} \frac{(u(x)-u(y))^2}{\abs{x-y}^{N+2s}}\,dx\,dy
\end{equation}
is a norm on the vector space $C^\infty_0(\Omega)$. 
The metric completion of this space is denoted, here and henceforth, by $\mathcal{D}^{s,2}_0(\Omega)$. 

\begin{remark}[Analogies and differences with other spaces]\label{rmk-analog}
Except for the special case $s=\frac12$,
if $\Omega$ is bounded with Lipschitz boundary, then $\mathcal{D}^{s,2}_0(\Omega)$ coincides with the closure $H_0^s(\Omega)$
of $C^\infty_0(\Omega)$ in the Sobolev-Slobodeckij space $H^s(\Omega)$ of all $u\in L^2(\Omega)$ such that
\[
	[u]^2_{H^s(\Omega)}\coloneqq\int_\Omega\int_\Omega  \frac{(u(x)-u(y))^2}{\abs{x-y}^{N+2s}}\,dx\,dy<+\infty
\]
In fact, in that case\footnote{See \cite[Appendix B]{BLP}.}, the ``censored'' Sobolev norm $\norm u_{L^2(\Omega)} + [u]_{H^s(\Omega)}$ is equivalent to 
\[
	\norm u_{L^2(\Omega)}+\tonda*{\int_{\R^N}\int_{\R^N} \frac{(u(x)-u(y))^2}{\abs{x-y}^{N+2s}}\,dx\,dy}^\frac12 
\]
and the latter is equivalent to the norm in $\mathcal{D}_0^{s,2}(\Omega)$, because Lipschitz sets support a Poincaré-type inequality.
On the contrary, if $\partial\Omega$ is not Lipschitz regular, then the existence of functions $u\in H^s(\Omega)$ for which the integral 
\[
	\int_\Omega\int_{\R^N\setminus\Omega}  \frac{u(x)^2}{\abs{x-y}^{N+2s}}\,dx\,dy
\]
diverges cannot be ruled out. If $\Omega$ is bounded and Lipschitz, then
$\X$ coincides with the Hilbert space $X_0^s(\Omega)=\graffa*{u\in H^s\tonda*{\R^N}:\text{$u=0$ a.e.\ in $\R^N\setminus\Omega$}}$ considered in \cite{MRS}.
\end{remark}

For a general open set, it is not true that all the elements of $\X$ are functions; $\X$ is not even a distribution space, in general (see, e.g., \cite{DL,HL}).
A restriction that clears off this difficulty 
is to consider open sets $\Omega$ supporting a Sobolev-type inequality, on which  $\X$ is a function space;
namely, assuming that
the infimum in \eqref{lambda1} is a positive number.

\subsection{Semilinear fractional spectrum}
We denote by $2^\ast_s$ the
fractional Sobolev conjugate exponent, defined by $2N/(N-2s)$ if $2s<N$ and 
$+\infty$ otherwise.

\begin{definition}[Semilinear fractional eigenvalues]\label{defil}
For $q\in(1,2^\ast_s)$, 
we consider the constrained critical points of the double integral \eqref{pre-1} along the submanifold
\begin{equation}
\label{submanif}
	\graffa*{u\in \X:\int_\Omega{\abs u^q\,dx}=1}
\end{equation}
We call \emph{$q$-semilinear $s$-eigenvalues} the corresponding constrained critical values. Their collection is denoted by $\mathfrak{S}(\Omega,s,q)$, and is said to be the \emph{$q$-semilinear $s$-spectrum} of $\Omega$.
\end{definition}

Clearly, \eqref{lambda1} is the largest lower bound
for $\mathfrak{S}(\Omega,s,q)$, and it is its minimum whenever the variational problem \eqref{lambda1} has a solution. The restriction $q<2^\ast_s$ in \cref{defil} is natural because for $q>2^\ast_s$ loss of compactness occur 
regardless of the properties of $\Omega$. If $0<s<N/2$, 
in the borderline case $q=2^\ast_s$
the infimum in \eqref{lambda1} is independent of $\Omega$, and gives the best
constant in Sobolev inequality, that reads as
\begin{equation}
\label{sobolev}
	\mathcal{S}(N,s) 
	\norm v_{L^{2^\ast_s}(\Omega)}^2
	\le
	\int_{\R^{N}}\int_{\R^N} \frac{(v(x)-v(y))^2}{\abs{x-y}^{N+2s}}\,dx\,dy \qquad \text{for all $v\in C^\infty_0\tonda*{\R^N}$}
\end{equation}

By Lagrange's multipliers rule, the $q$-semilinear $s$-eigenvalues are those 
positive real numbers 
$\lambda$ for which
\begin{equation}
\label{eq}
	(-\Delta)^s u =\lambda \norm u_{L^q(\Omega)}^{2-q}\abs u^{q-2}u
\end{equation}
has a non-trivial solution $u\in \mathcal{D}_0^{s,2}(\Omega)$ in the weak sense, viz.
\begin{equation}
\label{eqw}
	\int_{\R^N}\int_{\R^N} \frac{(u(x)-u(y))(\varphi(x)-\varphi(y))}{\abs{x-y}^{N+2s}}\,dx\,dy = \lambda \norm u_{L^q(\Omega)}^{2-q}\int_\Omega{\abs u^{q-2}u\varphi\,dx}
	\end{equation}
for all $\varphi\in \mathcal{D}^{s,2}_0(\Omega)$. 

\subsection{Fractional Lane-Emden equation}
After a renormalisation, the equation \eqref{eq} for Dirichlet $q$-semilinear $s$-eigenfunctions
becomes the \emph{fractional Lane-Emden equation}~\eqref{LE}. Given an open set $\mathscr U\subset\R^N$, we will say
a \emph{weak supersolution} (resp., \emph{subsolution}) of the latter in $\mathscr U$ any function $u\in \mathcal{D}_0^{s,2}(\mathscr U)$ such that
\begin{equation}
\label{LEqw}
	\int_{\R^N}\int_{\R^N} \frac{(u(x)-u(y))(\varphi(x)-\varphi(y))}{\abs{x-y}^{N+2s}}\,dx\,dy \ge
	\int_{\mathscr U}{\abs u^{q-2}u\varphi\,dx}\qquad
	\text{(resp., $\le$)}
\end{equation}
for all non-negative $\varphi\in \mathcal{D}^{s,2}_0(\mathscr U)$.
A function that is both a weak supersolution and a weak subsolution in $\mathscr U$ will be called a \emph{weak solution} in $\mathscr U$.
Clearly, the weak solutions of \eqref{LE} are the critical points on $\X$ of the free energy
\begin{equation}
\label{free-energy}
\frac12\int_{\R^N}\int_{\R^{N}}\frac{(\varphi(x)-\varphi(y))^2}{\abs{x-y}^{N+2s}}\,dx\,dy-\frac{1}{q}\int_\Omega{\abs\varphi^q\,dx}
\end{equation}

\begin{definition}[Fractional Lane-Emden densities]\label{defiw}
Let $q\in(1,2)$ and assume that
$\lambda_1(\Omega,s,q)>0$.
We denote by $w_{\Omega,s,q}$ the unique solution
of the variational problem
\begin{equation}
\label{defw}
	\min_{\varphi\in\X}\graffa*{\frac12\int_{\R^N}\int_{\R^N} \frac{(\varphi(x)-\varphi(y))^2}{\abs{x-y}^{N+2s}}\,dx\,dy - \frac{1}{q}\int_\Omega\varphi^q\,dx :
	\varphi\ge0\ \text{a.e.\ in $\Omega$}}
\end{equation}
and we call it the \emph{$(s,q)$--Lane-Emden density of $\Omega$}.

\begin{remark}\label{welldefi}
By \cite[Theorem 1.3]{F19}, the assumption $\lambda_1(\Omega,s,q)>0$ assures the compactness of the embedding
$\X\hookrightarrow L^q(\Omega)$; then, any minimising sequence
for \eqref{defw} is easily seen to be bounded in $\X$, so it
converges, up to relabelling, weakly in $\X$ and strongly in $L^q(\Omega)$. Also, the constraint $\varphi\ge0$ is convex. Thus, 
solutions of \eqref{defw} exist by direct methods in calculus of variations.
As for their uniqueness, 
minimisers of the even functional \eqref{free-energy} cannot change sign by \cref{lm:sign1}, and thence
constrained minimisers are non-negative minimisers of the free energy \eqref{free-energy}. Then, we conclude by the uniqueness of non-negative weak solutions of
\eqref{LE} (see \cref{uniw} below).
\end{remark}

\end{definition}

\section{The fractional semilinear spectral problem}\label{techn1}
Next proposition provides quantitative $L^\infty$-bounds for $q$-semilinear $s$-eigenfunctions $u$ corresponding to $\lambda\in\mathfrak{S}(\Omega,s,q)$ in terms of the $L^q(\Omega)$-norm of $u$ and of the eigenvalue $\lambda$.
For this standard result, in the proof we limit ourselves to check that Moser-type iterations such as those in appendix to \cite{BFR} can be repeated in this framework, too.

\begin{proposition}\label{linfty}
Let $q\in(1,2^\ast_s)$ and assume the
embedding $\X\hookrightarrow L^q(\Omega)$ to be compact. Let
$\lambda\in\mathfrak{S}(\Omega,s,q)$ and let $u\in \X$ be a corresponding $q$-semilinear $s$-eigenfunction. Then
\begin{subequations}
	\label{bound}
\begin{align}
	\norm u_{L^\infty(\Omega)} &\le
	\mathcal{C}_1(N,s,q) \lambda^{\frac{2^\ast_s}{2\tonda*{2^\ast_s-q}}} \norm u_{L^q(\Omega)} &
	\text{if $2^\ast_s<+\infty$}\\
	\norm u_{L^\infty(\Omega)} &\le
	\mathcal{C}_2(N,s,q,\abs\Omega) \lambda \norm u_{L^q(\Omega)} &
	\text{if $2^\ast_s=+\infty$}
\end{align}
\end{subequations}
\end{proposition}

\begin{proof}
With no loss of generality, we may assume that $u>0$.
Fix $\beta>1$ and $M>0$. By \cite[Lemma A.2]{BP} with $p=2$, $a=u(x)$, $b=u(y)$ and\footnote{From now on, we use the following notation:
\begin{align*}
    a\wedge b&\coloneqq\min\graffa{a,b}\\
    a\vee b&\coloneqq\max\graffa{a,b}
\end{align*}
} $g(t)=(t\wedge M)^\beta$, we get
\begin{multline}
\label{elineq}
\frac{2\beta}{\beta+1}\int_{\R^{N}}\int_{\R^N} \frac{\tonda*{(u(x)\wedge M)^{\frac{\beta+1}{2}}-(u(y)\wedge M)^{\frac{\beta+1}{2}}}^2}{\abs{x-y}^{N+2s}}\,dx\,dy\\
	\le \int_{\R^{N}}\int_{\R^N}\frac{(u(x)-u(y))\tonda*{(u(x)\wedge M)^\beta-(u(y)\wedge M)^\beta}}{\abs{x-y}^{N+2s}}\,dx\,dy
\end{multline}
The choice $\varphi=(u\wedge M)^\beta$ in \eqref{eqw} implies that the right integral in \eqref{elineq} does not exceed
\[
\lambda\norm u_{L^q(\Omega)}^{2-q} \int_\Omega u^{q-1}(u\wedge M)^\beta\,dx
\]

\paragraph{\bf Case $N>2s$}
By the compactness of the embedding 
$\X\hookrightarrow L^q(\Omega)$ and by a density argument,
Sobolev inequality \eqref{sobolev} holds with $v=(u\wedge M)^{\frac{\beta+1}{2}}$. Thus, the left hand side in \eqref{elineq} is at least
\[
\mathcal{S}(N,s) \frac{2\beta}{\beta+1} \tonda*{\int_\Omega (u\mathbin{\wedge}M)^{\frac{\beta+1}{2}2^\ast_s}\,dx}^\frac{2}{2^\ast_s}
\]
As $M>0$ was arbitrary, by the material above we deduce that
\begin{equation}\label{prebound}
\mathcal{S}(N,s) \tonda*{\int_\Omega u^{\frac{\beta+1}{2}2^\ast_s}\,dx}^\frac{2}{2^\ast_s}\le\lambda\norm u_{L^q(\Omega)}^{2-q}  \frac{\beta+1}{2\beta} \int_\Omega u^{\beta+q-1}\,dx
\end{equation}

If $1<q<2$, by arguing as in the second part of the proof of \cite[Proposition 2.5]{BF19} we see that  \eqref{prebound} implies \eqref{bound}. 
If instead $2\le q<2^\ast_s$, then, by H\"older's inequality, we have
\[
 \int_\Omega u^{\beta+q-1}\,dx \le \norm u_{L^q(\Omega)}^{q-2} \tonda*{\int_\Omega u^{\frac{\beta+1}{2}q}\,dx}^\frac{2}{q}
\]
whence it follows that
\[
	\mathcal{S}(N,s) \tonda*{\int_\Omega u^{\frac{\beta+1}{2}2^\ast_s}\,dx}^\frac{2}{2^\ast_s}\le\lambda  \frac{\beta+1}{2\beta}
	\tonda*{\int_\Omega u^{\frac{\beta+1}{2}q}\,dx}^\frac{2}{q}
\]
which leads one to \eqref{bound} again, thanks to the iteration scheme in first part of the proof of~\cite[Proposition 2.5]{BF19}.
\smallskip

\paragraph{\bf Case $N=1$ and $\frac12<s<1$}
In this case, the conclusion is an immediate consequence of fractional Morrey's
embedding (see \cite[Corollary 2.7]{BGCV}).
\smallskip

\paragraph{\bf Case $N=1$ and $s=\frac12$} 
The obvious fact
in this borderline case is that solutions have bounded mean oscillation. To prove they are also bounded, we first focus on exponents $q\in(1,2]$. By the second statement in \cite[Lemma 2.3]{F19}
with $p=2$, $N=1$ and $r=2q$, 
\begin{equation}
\label{GNS12}
C_1 \tonda*{\int_\Omega \varphi^{2q}\,dx}^\frac{2}{q}
\le \tonda*{\int_\Omega \varphi^q\,dx}^\frac{2}{q}
 \int_\R\int_\R
	\frac{(\varphi(x)-\varphi(y))^2}{\abs{x-y}^{2}}\,dx\,dy
\end{equation}
holds, in particular, with $\varphi=(u\wedge M)^{\frac{\beta+1}{2}}$, for all $M>0$. The constant $C_1>0$ depends only on $q$ and $s$. Then,
by \eqref{elineq}, arguing as done in the previous case we get
\[
C_2(s,q) \tonda*{\int_\Omega u^{\frac{\beta+1}{2} 2q}\,dx}^\frac{2}{2q}\le\lambda \norm u_{L^q(\Omega)}^{2-q} \frac{\beta+1}{2\beta} \int_\Omega u^{\beta+q-1}\,dx
\]
Hence, we arrive at the desired conclusion by arguing as done
after equation (13) in~\cite{BF19}, with minor changes (just replace $2^\ast$  by $2q$).

In order to deal with the exponents $q>2$, we take $\sigma\in\tonda*{\frac14,\frac12}$
with $\frac12-\sigma$ so small that
the Sobolev conjugate $2^\ast_\sigma=2/(1-2\sigma)$ exceeds $2q$ and we observe that, for all $\varphi\in \mathcal{D}_0^{s,2}(\Omega)$,
\[
	C_3\int_\R\int_\R
	\frac{(\varphi(x)-\varphi(y))^2}{|x-y|^{1+2\sigma}}\,dx\,dy
	\le
	\tonda*{\int_\Omega\varphi^2\,dx}^{2(1-2\sigma)}
	\tonda*{\int_\R\int_\R
	\frac{(\varphi(x)-\varphi(y))^2}{\abs{x-y}^{2}}\,dx\,dy}^{4\sigma}
\]
where $C_3$ is an absolute constant. This follows by a homogeneity argument based on the obvious remark that
\[
\iint_{\abs{y-x}<1} \frac{(\varphi(x)-\varphi(y))^2}{\abs{x-y}^{1+2\sigma}}\,dx\,dy
\le \iint_{\abs{y-x}<1}
\frac{(\varphi(x)-\varphi(y))^2}{\abs{x-y}^{2}}\,dx\,dy
\]
and 
\[
\iint_{\abs{y-x}\ge1}  \frac{(\varphi(x)-\varphi(y))^2}{\abs{x-y}^{1+2\sigma}}\,dx\,dy
\le 2 \int_\Omega\varphi(x)^2\int_{\abs{y-x}\ge1} \frac{dy}{\abs{x-y}^{1+2\sigma}}\,dx
\le \frac{2}{\sigma} \int_\Omega\varphi^2\,dx
\]
Recalling that $2<2q<2^\ast_\sigma$, by interpolation we also have
\[
	\tonda*{\int_\Omega \varphi^{2q}\,dx}^\frac{1}{2q}
	\le
	\tonda*{\int_\Omega \varphi^2\,dx}^\frac{\theta}{2}
	\tonda*{\int_\Omega \varphi^{2^\ast_\sigma}\,dx}^\frac{1-\theta}{2^\ast_\sigma}
\]
where $\theta\in(0,1)$. Then,
by Sobolev inequality \eqref{sobolev} with $\sigma$ instead of $s$
and by H\"older's inequality, we have again \eqref{GNS12}, 
but with a constant different from $C_1$, depending only on $\Omega$, $s$ and $q$.

In conclusion, we can take $\varphi=(u\wedge M)^{\frac{\beta+1}{2}}$
and argue as done for the exponents in the range $(1,2]$ to get the desired estimate also in the case $q>2$.
\end{proof}

The following elementary proposition contains a general property of the first
semilinear fractional eigenvalue. 

\begin{proposition}\label{prop:exist1}
Let $q\in(1,2^\ast_s)$ and assume the embedding $\X\hookrightarrow L^q(\Omega)$ to be compact.
Then, the infimum in \eqref{lambda1} is a minimum. Moreover, any minimiser
is either a strictly positive or a strictly negative function.
\end{proposition}

\begin{proof}
The existence of a minimiser is an immediate consequence of the direct methods in the calculus of variations. 
The fact that it must have constant sign follows by \cref{lm:sign1}. 
Then, the last statement follows by the strong minimum principle of \cref{nonlocmaxprinc}.
\end{proof}

Besides the first eigenvalue \eqref{lambda1}, higher eigenvalues also exist.
In fact, it is straightforward to check that the squared norm \eqref{1} in $\X$ satisfies
the Palais-Smale condition. Hence, in view of \cite[Theorem 5.7]{S}, 
$\mathfrak{S}(\Omega,q,s)$ is an infinite set.
More precisely, for all $n\in\mathbb N$
we denote by $\mathfrak{T}_n(\Omega,s,q)$ the collection of all subsets $A$ of 
\begin{equation}
\label{submanif4}
\graffa*{u\in\X:\int_\Omega{\abs u^q\,dx}=1}
\end{equation}
that are symmetric and compact in $\X$ and satisfy the following property; for every $k<n$, there exist no
odd and continuous mapping from $A$ to $\R^k\setminus\{0\}$. We can rephrase last property saying that the \emph{Krasnoselskii's genus} of $A$ is larger than or equal to $n$. Then, setting
\begin{equation}
\label{lambdan}
\lambda_n(\Omega,s,q) = \inf_{A \in\mathfrak{T}_n(\Omega,s,q)} \max_{u\in A}\int_{\R^N}\int_{\R^N} \frac{(u(x)-u(y))^2}{\abs{x-y}^{N+2s}}\,dx\,dy
\end{equation}
one defines an unbounded non-decreasing sequence of $q$-semilinear $s$-eigenvalues.

\begin{remark}\label{rmk:higher}
In general, $ \mathfrak{S}(\Omega,s,q)$ is closed. Indeed, if
a sequence $(\lambda_j)_{j\in\N}\subset   \mathfrak{S}(\Omega,s,q)$ converges to a positive number $\lambda$, there is a corresponding sequence
of $q$-semilinear $s$-eigenfunctions (obtained by renormalisation in $L^q(\Omega)$) which has constant $L^q(\Omega)$-norm and converging
norm in $\X$. By uniform convexity, some subsequence is converging strongly to a limit $u$ in $L^q(\Omega)$,
and this implies that $u$ is a $q$-semilinear $s$-eigenfunction corresponding to $\lambda$.
\end{remark}

\subsection{The sub-homogeneous case}
We recall two properties of $\lambda_1(\Omega,s,q)$ for $q\le2$.

\begin{proposition}\label{prop:sign2}
Let $q\in(1,2]$ and assume
that $\lambda_1(\Omega,s,q)>0$.
If $\lambda\in\mathfrak{S}(\Omega,s,q)$ and
$u$ is a corresponding eigenfunction, then $u\ge0$ a.e.\ in $\Omega$ implies $\lambda = \lambda_1(\Omega,s,q)$. 
\end{proposition}

\begin{proof}
By assumption, the embedding $\X\hookrightarrow L^q(\Omega)$ is continuous. Then, since $q\in(1,2]$, 
by Gagliardo-Nirenberg interpolation inequality (see \cite[Lemma 2.3]{F19}) it is also compact.
Thus, the assumptions of \cref{prop:exist1} are valid.

Let $v\in\X$ be a first eigenfunction, and assume that $v>0$ a.e.\ in $\Omega$.
Then, let $\lambda \in\mathfrak{S}(\Omega,s,q)$, let $u$ be a  corresponding eigenfunction, and assume that $u\ge0$ a.e.\ in $\Omega$, as well. This implies
$u>0$ a.e.\ in $\Omega$ by the strong minimum principle (\cref{nonlocmaxprinc}). Being free to multiply by constants,
we shall also assume both $u$ and $v$ to have unit norm in $L^q(\Omega)$.

Fix $\varepsilon>0$ and write $u_\varepsilon = u+\varepsilon$. 
For every $x,y\in\R^N$, by~\cite[Proposition 4.2]{BF14} with $p=2$, we have
\[
(u(x)-u(y))\tonda*{\frac{v(x)^q}{u_\varepsilon(x)^{q-1}}-\frac{v(y)^q}{u_\varepsilon(y)^{q-1}}}\le
\abs{v(x)-v(y)}^q\abs{u(x)-u(y)}^{2-q}
\]
Multiplying by the kernel $\abs{x-y}^{N+2s}=\abs{x-y}^{N\frac{q}{2}+sq+N\tonda*{1-\frac{q}{2}}+s(2-q)}$ and integrating yields
\begin{multline*}
	\int_{\R^{N}}\int_{\R^N}\frac{u(x)-u(y)}{\abs{x-y}^{N+2s}} \tonda*{ \frac{v(x)^q}{u_\varepsilon(x)^{q-1}}-\frac{v(y)^q}{u_\varepsilon(y)^{q-1}}}\,dx\,dy\\
	\le \int_{\R^{N}}\int_{\R^N}\frac{\abs{v(x)-v(y)}^{q}\abs{u(x)-u(y)}^{2-q}}{\abs{x-y}^{(N+2s)\frac q2}\abs{x-y}^{(N+2s)\frac{2-q}2}}\,dx\,dy
\end{multline*}
By H\"older's inequality with exponents $\frac{2}{q}$ and $\frac{2}{2-q}$, the right hand side is bounded by
\[
\lambda_1(\Omega,s,q)^{\frac{q}{2}}\lambda^\frac{2-q}{2}
\]
because of the equations satisfied by $u$ and $v$ and of their normalisation in $L^q(\Omega)$.
Since $\varphi =v^q/u_\varepsilon^{q-1}$ is an admissible test function in \eqref{eqw}, we have
\[
	\int_{\R^{N}}\int_{\R^N}\frac{u(x)-u(y)}{\abs{x-y}^{N+2s}}\tonda*{ \frac{v(x)^q}{u_\varepsilon(x)^{q-1}}-\frac{v(y)^q}{u_\varepsilon(y)^{q-1}}}\,dx\,dy 
	=\lambda \int_\Omega u(x)^{q-1} \frac{v(x)^{q}}{(u(x)+\varepsilon)^{q-1}}\,dx
\]
Therefore, for every $\varepsilon>0$ we end up with inequality
\begin{equation}
\label{prefatou}
\lambda \int_\Omega u(x)^{q-1} \frac{v(x)^{q}}{(u(x)+\varepsilon)^{q-1}}\,dx\le
\lambda_1(\Omega,s,q)^{\frac{q}{2}}\lambda^\frac{2-q}{2}
\end{equation}
Since $u>0$ a.e.\ in $\Omega$, applying Fatou's lemma and dividing $\lambda$ out we arrive at
\[
1=\int_\Omega v(x)^q\,dx \le  \tonda*{\frac{\lambda_1(\Omega,s,q)}{\lambda}}^{\frac{q}{2}}
\]
which gives $\lambda\le \lambda_1(\Omega,s,q)$. The definition of $\lambda_1(\Omega,s,q)$ gives the opposite inequality.
\end{proof}

\begin{proposition}\label{prop:simpl}
Let $q\in(1,2]$ and assume that $\lambda_1(\Omega,s,q)>0$.
Then,
$\lambda_1(\Omega,s,q)$ is simple, i.e., all the corresponding eigenfunctions are mutually proportional.
\end{proposition}

\begin{proof}
Let $u$ and $v$ be first eigenfunctions. With no loss of generality, assume
that both $u$ and $v$ are non-negative functions. 
We may also assume both $u$ and $v$ to have unit norm in $L^q(\Omega)$.
For all $t\in[0,1]$, consider the function $\xi_t\colon\Omega\to\R^2$ defined by $\xi_t(x)=\tonda*{t^{1/q}u(x),(1-t)^{1/q}v(x)}$.
Let $\|\cdot\|_{\ell^q}$ denote the $\ell^q$-norm in $\R^2$. Then, the convexity of  $\tau\mapsto\abs\tau^{2/q}$ implies
\begin{subequations}\label{triangle12}
\begin{equation}
\label{triangle1}
	\norm{\xi_t(x)-\xi_t(y)}_{\ell^q}^2 \le t (u(x)-u(y))^2+(1-t)(v(x)-v(y))^2 \qquad \text{for all $x,y\in\Omega$}
\end{equation}
Also, for every $t\in[0,1]$, set $\sigma_t(x) = \norm{\xi_t(x)}_{\ell^q}$ for $x\in\Omega$ and $\sigma_t(x)=0$ for $x\in\R^N\setminus\Omega$.
Then
\begin{equation}
\label{triangle2}
(\sigma_t(x)-\sigma_t(y))^2= 
(\norm{\xi_t(x)}_{\ell^q}-\norm{\xi_t(y)}_{\ell^q})^2
	\qquad \text{for all $x,y\in\Omega$}
\end{equation}
\end{subequations}
Hence, by triangle inequality, $\sigma_t\in\X$ 
with the estimate
\[
\iint_{\R^{2N}}\frac{(\sigma_t(x)-\sigma_t(y))^2}{\abs{x-y}^{N+2s}}dx\,dy\le t\iint_{\R^{2N}}\frac{(u(x)-u(y))^2}{\abs{x-y}^{N+2s}}dx\,dy+(1-t)\iint_{\R^{2N}}\frac{
(v(x)-v(y))^2}{\abs{x-y}^{N+2s}}dx\,dy
\]
The normalisation in $L^q(\Omega)$ of $u$ and of $v$ implies that the right hand side in the latter equals $\lambda_1(\Omega,s,q)$. On the other hand,
the left hand side is larger than or equal to $\lambda_1(\Omega,s,q)$, because
 \[
 \int_\Omega\sigma_t(x)^q\,dx= \int_\Omega {\norm{\xi_t(x)}_{\ell^q}^q\,dx}=t\int_\Omega u(x)^q\,dx+(1-t)\int_\Omega v(x)^q\,dx=1
 \]
thus, $\sigma_t$ is admissible for the minimisation problem that defines $\lambda_1(\Omega,s,q)$. Therefore, for every $t\in[0,1]$, the
previous integral inequality is an equality. As a consequence, the pointwise identity 
\[
(\sigma_t(x)-\sigma_t(y))^2=t (u(x)-u(y))^2+(1-t)(v(x)-v(y))^2
\]
holds for all $t\in[0,1]$ and for a.e.\ $x,y\in\Omega$. In view of \eqref{triangle12}, the latter yields the equality case in triangle inequality
\[
\abs{\norm{\xi_t(x)}_{\ell^q}-\norm{\xi_t(y)}_{\ell^q}} \le \norm{\xi_t(x)-\xi_t(y)}_{\ell^q}
\]
which occurs if and only if there exists $\alpha(x,y)\in\R$ with $\xi_t(x)=\alpha(x,y)\xi_t(y)$. Owing to the definition of $\xi_t$, it
follows that $u(x)=\alpha(x,y)u(y)$ and $v(x)=\alpha(x,y)v(y)$. In conclusion, for a.e.\ $x,y$, we have
\[
	\frac{u(x)}{v(x)}=\frac{u(y)}{v(y)}
\]
and this concludes the proof.
\end{proof}

\subsection{The super-homogeneous case}
Following the proof of \cite[Proposition 4.3]{BF19} about an analogous property in the local case,
we show that the first eigenvalue on  $\Omega$ is simple
also in the super-homogeneous case $q>2$, for all $q$ up to a suitable threshold (depending on $\Omega$). For this purpose, we first discuss
the continuous dependence of $\lambda_1(\Omega,s,q)$ on $q$ with a method used in \cite[Lemma 4]{anello} to derive monotonicity of semilinear eigenvalues with respect to $q$ in the local case; here we limit our attention to the right continuity at $q=2$, which can be proved also by different methods (see \cite[Lemma 2.1]{B}).

\begin{lemma}\label{lambdacont}
We have
\[
	\lim_{q\to 2^+}\lambda_1(\Omega,s,q)=\lambda_1(\Omega,s,2)
\]
\end{lemma}
\begin{proof}
By \cite[Corollary 1.2]{F19}, we have $\lambda_1(\Omega,s,2)>0$ if and only if
\begin{equation}
\label{immersioni}
\lambda_1(\Omega,s,q)>0 \qquad\text{for every $q\in[2,2^\ast_s)$}
\end{equation}
Hence, we can assume that \eqref{immersioni} holds, otherwise the conclusion is obvious. Therefore, 
\begin{equation}
\label{reciproco}
	\sup_{v\in C^\infty_0(\Omega)}\graffa*{
	\int_\Omega{\abs v^q\,dx}:
	\int_{\R^{N}}\int_{\R^N}\frac{(v(x)-v(y))^2}{\abs{x-y}^{N+2s}}\,dx\,dy=1}= \lambda_1(\Omega,s,q)^{-\frac q2}
\end{equation}
for all $q\in[2,2^\ast_s)$,
which can be seen by a straightforward homogeneity argument. Since
\[
	\diff*[2]{\int_\Omega{\abs v^q\,dx}}q = \int_{\{v\ne0\}}{\abs v^q(\log{\abs v})^2\,dx}\ge0
	\qquad \text{for all $q>1$}
\]
the left hand side of \eqref{reciproco}, as a function of $q$, is the pointwise supremum of a family of
lower semicontinuous convex functions on $(1,2^\ast_s)$. Thus, $q\mapsto\lambda_1(\Omega,s,q)^{-q/2}$
is continuous on $[2,2^\ast_s)$, and thence so it is 
$q\mapsto\lambda_1(\Omega,s,q)$ 
on $[2,2^\ast_s)$, by composition.
\end{proof}

\begin{proposition}\label{3.2}
Assume that the embedding $\X\hookrightarrow L^2(\Omega)$ is compact. Then, there exists $q_\Omega\in(2,2^*_s)$ such that $\lambda_1(\Omega,s,q)$ is simple for all $q\in(2,q_\Omega)$.
\end{proposition}

\begin{proof}
Let $(q_n)_{n\in\mathbb N}$ be a decreasing sequence
converging to $2$ and let $(u_n)_{n\in\mathbb N}$ and $(v_n)_{n\in\mathbb N}$ be sequences in $\X$ such that, for all $n\in\mathbb N$,  equation
\eqref{eqw} holds with $\lambda = \lambda_1(\Omega,s,q_n)$ both for
$u=u_n$ and for $u=v_n$. By \cref{prop:exist1}, we may assume
$u_n$ and $v_n$ to be positive functions, nor does it cause any loss of generality assuming them to have unit $L^{q_n}(\Omega)$-norm. Then, by using themselves as test functions in their own equations, in view of \cref{lambdacont} we see that
\[
\lim_{n\to\infty}
\iint_{\R^{2N}} \frac{(u_n(x)-u_n(y))^2}{\abs{x-y}^{N+2s}}\,dx\,dy
=
\lim_{n\to\infty}
\iint_{\R^{2N}} \frac{(v_n(x)-v_n(y))^2}{\abs{x-y}^{N+2s}}\,dx\,dy
=\lambda_1(\Omega,s,2)
\]
Also, because, by assumption, the infimum that defines
$\lambda_1(\Omega,s,2) $ is achieved, we have
\[
\lambda_1(\Omega,s,2) 
=\int_{\R^{N}}\int_{\R^N} \frac{(\bar u(x)-\bar u(y))^2}{\abs{x-y}^{N+2s}}\,dx\,dy
\]
for an appropriate function $\bar u\in\X$ with unit norm in $L^2(\Omega)$. 

By \cref{prop:simpl}, $\bar u$ is uniquely determined; hence, from the assumption that the embedding $\X\hookrightarrow L^2(\Omega)$ is compact, we infer that both $(u_n)_{n\in\mathbb N}$
and $(v_n)_{n\in\mathbb N}$ converge to $\bar u$ strongly in $\X$ and pointwise a.e.\ in $\Omega$,
by using the last two identities in display and the fact that, for any given $\gamma>2$, owing to \cref{linfty} we have
\begin{align*}
\norm{u_n-\bar u}_{L^\gamma(\Omega)}&\le  c\norm{u_n-\bar u}_{L^2(\Omega)}^{\frac2\gamma}\\
	\norm{v_n-\bar u}_{L^\gamma(\Omega)}&\le c
	\norm{v_n-\bar u}_{L^2(\Omega)}^{\frac2\gamma}
\end{align*}
for a constant $c>0$ independent of $n$.

As $q_n>2$, by \cref{linfty}
there is a constant $C$, depending only on the data, such that
\begin{equation}
\label{stima}
    w_n\coloneqq(q_n-1)\int_0^1 [tu_n+(1-t)v_n]^{q_n-2}\,dt \le C
\end{equation}
The latter appears as a weight in the equation for $\psi_n=\norm{u_n-v_n}_{L^2(\Omega)}^{-1}(u_n-v_n)$, viz.
\begin{equation}
\label{0prelim1}
    \int_{\R^{N}}\int_{\R^N} \frac{(\psi_n(x)-\psi_n(y))(\phi(x)-\phi(y))}{\abs{x-y}^{N+2s}}\,dx\,dy
    =\lambda_1(\Omega,s,q_n)\int_{\Omega}w_n\psi_n\phi\,dx
\end{equation}
for all $\phi\in\X$. After choosing $\varphi=\psi_n$ in \eqref{0prelim1}, in view of \eqref{stima} we see that $(\psi_n)_{n\in\N}$
is bounded in $\X$.  Thus, by assumption, a subsequence (not relabelled) converges to some limit $\psi$, weakly in $\X$ and strongly in $L^2(\Omega)$. Then, $\psi$ is bound to have unit norm in $L^2(\Omega)$, in particular $\psi\neq0$.
\vskip.2cm

We claim that
\begin{equation}
\label{l2loclim}
w_n\to1\quad\text{in $L^2_{\rm loc}(\Omega)$}
\end{equation}
Thence, recalling also \cref{lambdacont}, by passing to the limit in \eqref{0prelim1} we arrive at 
\begin{equation}
\label{0postlim1}
\int_{\R^{N}}\int_{\R^N} \frac{(\psi(x)-\psi(y))(\phi(x)-\phi(y))}{\abs{x-y}^{N+2s}}\,dx\,dy
    =\lambda_1(\Omega,s,2)\int_{\Omega}\psi\phi\,dx
\end{equation}
for all $\phi\in\X$, i.e., $\psi$ is a non-trivial first eigenfunction. By \cref{prop:simpl}, it follows that either $\psi=\bar u$ or $\psi=-\bar u$. On the other hand,
we can plug in $\varphi=\psi_n^\pm$ into \eqref{0prelim1} and deduce from \eqref{stima}, for $n$ large enough, that
\[
	\int_{\R^{N}}\int_{\R^N}
	 \frac{(\psi_n^\pm(x)-\psi_n^\pm(y))^2}{\abs{x-y}^{N+2s}}\,dx\,dy
    \le 2C\lambda_1(\Omega,s,2)\int_{\Omega}{\abs*{\psi_n^\pm}}^2\,dx
\]

We argue by contradiction and we assume that $u_n\neq v_n$, for all $n\in\N$. Hence, both
 $\Omega_n^+=\{u_n>v_n\}$ and $\Omega_n^-=\{u_n< v_n\}$
must have non-zero measure, because $u_n$ and $v_n$ have the same $L^q(\Omega)$-norm. Then, we can estimate from below the left hand side to get
\[
\abs*{\Omega_n^\pm}^\frac{2s}{N}\int_{\R^{N}}\int_{\R^N} 
	 \frac{(\psi_n^\pm(x)-\psi_n^\pm(y))^2}{\abs{x-y}^{N+2s}}\,dx\,dy	
	 \ge C' \int_\Omega{\abs*{\psi_n^\pm}}^2\,dx
\]
where $C'$ depends only on $N$ and $s$;
indeed, if $\Omega_n^\pm$ has infinite measure, then the latter is trivial; otherwise, we can deduce it
from the definition of $\lambda_1(\Omega,s,2)$,
its scaling properties and the fractional Faber-Krahn inequality (see \cite[Theorem 3.5]{BLP}).
Combining the upper and the lower bound yields
\(
	\inf_{n\in\N}{\abs*{\Omega_n^\pm}}>0
\),
which is inconsistent with the pointwise convergence of
$\psi_n$ to its constant sign limit $\psi$.

Thus, we are left with proving the claim \eqref{l2loclim}.
To do so,
we consider a bounded open set $\Omega'\Subset\Omega$ and observe that
\begin{equation*}
    \begin{split}
       \int_{\Omega'}(w_n-1)^2\,dx&=\int_{\Omega'}\tonda*{\int_0^1(q_n-1)[tu_n+(1-t)v_n]^{q_n-2}\,dt-1}^2\,dx\\
       &\le\int_{\Omega'}\int_0^1\quadra*{(q_n-1)[tu_n+(1-t)v_n]^{q_n-2}-1}^2\,dt\,dx\\
       &\le2(q_n-2)^2\int_{\Omega'}\int_0^1\tonda*{[tu_n+(1-t)v_n]^{q_n-2}}^2\,dt\,dx\\
       &+2\int_{\Omega'}\int_0^1\tonda*{[tu_n+(1-t)v_n]^{q_n-2}-1}^2\,dt\,dx\\
       &\le 2 C^2\abs*{\Omega'}+2\int_{\Omega'}\int_0^1\tonda*{\abs{tu_n+(1-t)v_n}^{q_n-2}-1}^2\,dt\,dx
    \end{split}
\end{equation*}
By the pointwise convergence a.e.\ in $\Omega$
of both $u_n$ and $v_n$ to $\bar u$
and by \eqref{stima}, the latter implies
that $w_n\to1$ in $L^2(\Omega')$ by dominated convergence theorem.
Since $\Omega'$ was arbitrary, that entails \eqref{l2loclim}, as desired.
\end{proof}

\section{Fractional Lane-Emden densities}\label{wfrac}

In this section we always limit our attention to exponents $q\in(1,2)$ and
we prove some properties of the fractional Lane-Emden density of $\Omega$.
We recall that in this paper the function 
$w_{\Omega,s,q}$ is introduced in \cref{defiw}, under the assumption that
$\lambda_1(\Omega,s,q)>0$, as a non-negative weak solution of \eqref{LE} (see also \cref{welldefi}).

\begin{remark}\label{uniw}
Equation \eqref{LE} has indeed a unique non-negative weak solution; by \cref{prop:sign2}, any such function is
a non-negative $q$-semilinear $s$-eigenfunction with $L^q(\Omega)$-norm equal to $\lambda_1(\Omega,s,q)^{\frac1{q-2}}$, whence the uniqueness by \cref{prop:simpl}.
\end{remark}

\begin{proposition}\label{CP}
Let $q\in(1,2)$, let $\Omega_1$ and $\Omega_2$ be bounded open sets and, for $i\in\{1,2\}$, let $w_i $ be the fractional Lane-Emden density $w_{\Omega_i,s,q}$ on $\Omega_i$. Then 
\[
\Omega_1\subset \Omega_2 \implies w_1\le w_2
\]
\end{proposition}

\begin{proof}
Let us write $w_i = w_{\Omega_i,s,q}$ in $\Omega_i$ and $w_i=0$ in 
$\mathbb R^N\setminus \Omega_i$, for $i\in\{1,2\}$.
The inequality
\[
	(a\vee b-c\vee d)^2- (a-c)^2
	\le (b-d)^2- (a\wedge b - c\wedge d)^2 
\]
with $a=w_1(x)$, $b=w_2(x)$, $c=w_1(y)$ and $d=w_2(y)$ entails
the submodularity property
\begin{multline*}
	\frac12\iint_{\mathbb R^{2N}}\frac{(
		(w_1\vee w_2)(x)-(w_1\vee w_2)(y))^2}{\abs{x-y}^{N+2s}}\,dx\,dy
	- 	\frac12\iint_{\mathbb R^{2N}} \frac{(
		w_1(x)-w_1(y))^2}{\abs{x-y}^{N+2s}}\,dx\,dy\\
	\le
	\frac12\iint_{\mathbb R^{2N}} \frac{(
		w_2(x)-w_2(y))^2}{\abs{x-y}^{N+2s}}\,dx\,dy-
		\frac12\iint_{\mathbb R^{2N}} \frac{(
		(w_1\wedge w_2)(x)-(w_1\wedge w_2)(y))^2}{\abs{x-y}^{N+2s}}\,dx\,dy
\end{multline*}
By minimality of $w_1$, we also have
\begin{multline*}
\frac12\int_{\mathbb R^N}\int_{\mathbb R^N} \frac{(
		w_1(x)-w_1(y))^2}{\abs{x-y}^{N+2s}}\,dx\,dy
	-\frac1q \int_{\Omega_1} w_1^q\,dx\\
	\le\frac12\int_{\mathbb R^N}\int_{\mathbb R^N} \frac{(
		(w_1\wedge w_2)(x)-(w_1\wedge w_2)(y))^2}{\abs{x-y}^{N+2s}}\,dx\,dy
	-\frac1q \int_{\Omega_1} (w_1\wedge w_2)^q\,dx
\end{multline*}
Taking into account the integral identity
\[
		\frac1q\int_{\Omega_1} w_1^q\,dx
		-\frac1q
	\int_{\Omega_2} (w_1\vee w_2)^q\,dx
	=
	\frac1q\int_{\Omega_1} (w_1\wedge w_2)^q\,dx
	-\frac1q\int_{\Omega_2} w_2^q\,dx
\]
and summing up, then, gives
\begin{multline*}
	\frac12\int_{\mathbb R^N}\int_{\mathbb R^N} \frac{(
		(w_1\vee w_2)(x)-(w_1\vee w_2)(y))^2}{\abs{x-y}^{N+2s}}\,dx\,dy-\frac1q\int_{\Omega_2} (w_1\vee w_2)^q\,dx\\ 
	\le
	\frac12\int_{\mathbb R^N}\int_{\mathbb R^N} \frac{(
		w_2(x)-w_2(y))^2}{\abs{x-y}^{N+2s}}\,dx\,dy-\frac1q\int_{\Omega_2}  w_2^q\,dx
\end{multline*}
Hence, by the minimality property of $w_2$, we infer that
$w_2= w_1\vee w_2$, as desired.
\end{proof}

We can extend \cref{defiw} to the case $\lambda_1(\Omega,s,q)=0$, as done in the local case (see~\cite{BFR}).

\begin{definition}\label{wqs}
Let $q\in(1,2)$. 
Then, we set
\begin{equation}
\label{defw3}
	w_{\Omega,s,q}(x) =\lim_{r\to\infty}w_{\Omega\cap B_r,s,q}(x) \qquad \text{for all $x\in\Omega$}
\end{equation} 
and we continue to call $w_{\Omega,s,q}$ the $(s,q)$--Lane-Emden density of $\Omega$.
\end{definition}

By \cref{CP}, the limit \eqref{defw3} always exists, so that
the definition is well posed. The following lemma assures its consistency with \cref{defiw}.

\begin{lemma}
Let $q\in(1,2)$ and assume that
$\lambda_1(\Omega,s,q)>0$.
For every $r>0$, we set $w_r(x)=w_{\Omega\cap B_r,s,q}(x)$ if $x\in B_r$ and $w_r(x)=0$ otherwise. Then,
$w_r$ converge pointwise to $w_{\Omega,s,q}$ as $r\to+\infty$.
\end{lemma}

\begin{proof} 
As $r\to+\infty$, the $(s,q)$--Lane-Emden density $w_r$ on $\Omega\cap B_r$
converges to an appropriate function  $\overline w\le  w_{\Omega,s,q}$.
By minimality, for every given $\varphi \in C^\infty_0(\Omega)$ there exists $R_\varphi>0$ such that, for all $r\ge R_\varphi$, we have
\begin{multline}
\label{minwr}
	\frac12\int_{\R^N}\int_{\R^N} \frac{(w_r(x)-w_r(y))^2}{\abs{x-y}^{N+2s}}\,dx\,dy - \frac1q\int_{\Omega\cap B_r}w_r^q\,dx\\
	\le\frac12\int_{\R^N}\int_{\R^N} \frac{(\varphi(x)-\varphi(y))^2}{\abs{x-y}^{N+2s}}\,dx\,dy - \frac1q\int_{\Omega\cap B_r}{\abs\varphi^q\,dx}
\end{multline}
Note that the equation for $w_r$ is \eqref{eqw} with $\Omega\cap B_r$ in place
of $\Omega$, $u=w_r$ and $\lambda = \norm{w_r}_{L^q(\Omega\cap B_r)}^{q-2}$.
Testing with $\varphi=w_r$ the equation for $w_r$, we get
\[
\begin{split}
	\int_{\R^N}\int_{\R^N} \frac{(w_r(x)-w_r(y))^2}{\abs{x-y}^{N+2s}}\,dx\,dy & = \int_{\Omega\cap B_r}w_r^q\,dx\\ 
	& \le
	\lambda_1(\Omega,s,q)^{\frac{q}{2}}\tonda*{\int_{\R^{N}}\int_{\R^N} \frac{(w_r(x)-w_r(y))^2}{\abs{x-y}^{N+2s}}\,dx\,dy}^\frac{q}{2}
\end{split}
\]	
where in the second inequality we also used that $w_r=0$ in $\Omega\setminus B_r$. Since $q<2$,
we deduce that $w_r$ converges to $\overline w$ weakly in $\X$ and strongly in $L^q(\Omega)$. Thus, passing to the limit as $r\to\infty$ in \eqref{minwr},
we obtain
\[
	\frac12\iint_{\R^{2N}}\frac{(\overline w(x)-\overline w(y))^2}{\abs{x-y}^{N+2s}}\,dx\,dy - \frac1q\int_\Omega\overline w^q\,dx
	\le \frac12\iint_{\R^{2N}} \frac{(\varphi(x)-\varphi(y))^2}{\abs{x-y}^{N+2s}}\,dx\,dy - \frac1q\int_\Omega{\abs\varphi^q\,dx}
\]
for all $\varphi\in C^\infty_0(\Omega)$, which, by uniqueness, implies that $\overline w = w_{\Omega,s,q}$.
\end{proof}

Following \cite{Knonloc}, for all $w\in\mathcal D^{s,2}_0\tonda*{\R^N}$ and for all $x_0\in\R^N$, we set
\[
\Tail(w,x_0,\rho)=\rho^{2s}\int_{\R^N\setminus B_\rho(x_0)}\frac{\abs{w(x)}}{\abs{x-x_0}^{N+2s}}\,dx
\]
The only difference between next proposition and \cite[Theorem 1.1]{Knonloc} is that we consider a non-homogeneous equation. We present the proof of \cite{Knonloc} for sake of completeness. Clearly, a similar estimate holds for non-homogeneous equations with data in $L^\gamma$, $\gamma>N/s$, but that is not relevant to our case.

\begin{proposition}\label{lm:subsoLE}
Let $\mathscr U\subset\R^N$ be an open set, $x_0\in\mathscr U$, $\delta\in(0,1]$, $0<r<\dist(x,\partial\mathscr U)$, 
$f\in L^{\infty}(\mathscr U)$
and let $w\in \mathcal{D}_0^{s,2}\tonda*{\R^N}$ be a non-negative weak subsolution of
$(-\Delta)^s w=f$ in $\mathscr U$, i.e., $w\ge0$ in $\mathscr U$ and
\begin{equation*}
\int_{\R^{N}}\int_{\R^N} \frac{(w(x)-w(y))(\phi(x)-\phi(y))}{\abs{x-y}^{N+2s}}
\,dx\,dy\le \int_{\mathscr U} f w \varphi\,dx
\end{equation*}
for all non-negative $\varphi\in C^\infty_0(\mathscr U)$. Then
\begin{equation*}
\ess_{B_{r/2}(x_0)}w\le C\quadra*{
\delta\Tail(w,x_0,r/2)  + 
\delta r^{2s}\norm f_{L^\infty(\mathscr U)}
+ \tonda*{\frac{ r^{N-2s}}{\delta}}^\frac{N}{4s} \tonda*{\intmed_{B_r(x_0)} w^2\,dx}^\frac12}
\end{equation*}
where the constant depends only on $N$ and $s$.
\end{proposition}

\begin{proof}
Let $r_k=\frac r2\tonda*{1+2^{-k}}$, $\tilde r_k=(r_{k+1}+r_k)/2$, $B_k=B_{r_k}(x_0)$ and $\tilde B_k = B_{\tilde r_k}(x_0)$.
We take $h>0$ and we define $h_k=\tonda*{1-2^{-k}}h$
and $\tilde h_k=(h_k+h_{k+1})/2$.
We take a cut-off function $\zeta_k\in C^\infty_0\tonda*{\tilde B_k}$,
with $\abs{\nabla\zeta_k}\le 2^{k+1}r^{-1}$, from $B_{k+1}$ to $\tilde B_k$.
We set $w_k=(w-h_k)_+$ and $\tilde w_k=\tonda*{w-\tilde h_k}_+$ and we observe that, by Minkowski's inequality and
fractional Poincaré-Sobolev inequality, there exists an absolute constant $C_0>0$ with
\begin{equation}
\label{T6.1}
\tonda*{
\norm{\tilde w_k\zeta_k}_{L^{2^\ast_s}(B_k)}-\intmed_{B_k} w_k\zeta_k\,dx}^2
\le \frac{C_0}{r^{N-2s}}
\int_{B_k}\int_{B_k} \frac{(\tilde w_k(x)\zeta_k(x)-\tilde w_k(y)\zeta_k(y))^2}{\abs{x-y}^{N+2s}}\,dx\,dy
\end{equation}
By the fractional Caccioppoli inequality (see \cite[Proposition 3.5]{BP}), the
right hand side in \eqref{T6.1} must not exceed
$C_1\mathcal{I}_1+C_2\mathcal{I}_2+C_3\mathcal{I}_3$, where
$C_1,C_2,C_3$ are constants depending only on $N$ and $s$ and 
\begin{align*}
\mathcal{I}_1&=r^{-N+2s} \int_{B_k}\int_{B_k}
\frac{(\zeta_k(x)-\zeta_k(y))^2}{|x-y|^{N+2s}}
\tonda*{\tilde w_k(x)^2+\tilde w_k(y)^2}\,dx\,dy\\
\mathcal{I}_2&=r^{-N+2s}\tonda*{
\sup_{y\not\in \tilde B_k} \int_{\R^N\setminus B_{r/2}(x_0)}
\frac{\tilde w_k(x)}{\abs{x-y}^{N+2s}}\,dx}\int_{B_k}  \tilde w_k\zeta_k^2\,dx\\
\mathcal{I}_3&= r^{-N+2s}\int_{B_k} f \tilde w_k\zeta_k^2\,dx
\end{align*}

In order to estimate the sum of these three terms, we set
\begin{equation}
\label{T666}
	Y_k=\tonda*{\int_{B_k}w_k^2\,dx}^\frac12
\end{equation}
Recalling that $\abs{\nabla\zeta_k}^2\le r^{-2}4^{k+2}$, $0\le\zeta_k\le1$ and $\tilde w_k\le w_k$, it is easily seen that
\begin{subequations}
\label{t61}
\begin{equation}
	\sqrt{C_1\mathcal{I}_1}\le r^{-\frac N2} 2^k Y_k
\end{equation}
Since $\frac{\abs{x-x_0}}{\abs{x-y}}\le \frac{\abs{x-x_0}}{\abs{x-x_0}-\abs{x_0-y}}\le 2^{k+1}$
for all $x\in \R^N\setminus B_{r/2}(x_0)$
and $y\in \tilde B_{k}$, we also see that
\begin{equation}
\begin{split}
\sqrt{C_2\mathcal{I}_2} & \le \frac{2^{\frac{N+2s}{2}k}}{r^{N/2}}\Tail(w,x_0,r/2)^\frac12
\tonda*{\int_{B_k}\tilde w_k\zeta_k^2\,dx}^\frac12
\\ &
\le 2\frac{2^{\frac{N+2s+1}{2}k}}{r^{N/2}}\tonda*{\frac{\Tail(w,x_0,r/2)}{h}}^\frac12 Y_k
\end{split}
\end{equation}
where in the last inequality we also used that
$\tilde w_k\zeta_k^2 \le 2^k(4/h)w_k^2$ in $B_k\cap\graffa*{w\ge \tilde h_k}$.
Similarly, we also have
\begin{equation}
	\sqrt{C_3\mathcal{I}_3}
	\le r^{-\frac{N}{2}+s} \norm f_{L^\infty(\mathscr U)}^{1/2} \tonda*{\int_{B_k}\tilde w_k \zeta_k^2\,dx}^\frac{1}{2}
	\le 2r^{-\frac{N}{2}+s}2^k \frac{\norm f_{L^\infty(\mathscr U)}^{1/2}}{\sqrt{ h}}Y_k
\end{equation}
\end{subequations}

The elementary inequality $\sqrt{a+b+c}\le \sqrt{a}+\sqrt{b}+\sqrt{c}$ for positive numbers $a,b,c$, the fact that
$(C_1\mathcal{I}_1+C_2\mathcal{I}_2+C_3\mathcal{I}_3)^{1/2}$
is an upper bound for the right hand side in \eqref{T6.1} and the
inequalities \eqref{t61} imply
\[
\norm{\tilde w_k\zeta_k}_{L^{2^\ast_s}(B_k)}
\le C_4 r^{-\frac N2}Y_k
\quadra*{
2^k+2^{\frac{N+2s+1}{2}k} \tonda*{
\frac{\Tail(w,x_0,r/2)}{ h}}^\frac12 
+2^k  r^{s} \frac{\norm f_{L^\infty(\mathscr U)}^{1/2}}{\sqrt{h}}}
\]
On the other hand, setting $\alpha=2s/N$ and $\beta=2s/(N-2s)$, we have
\[
	\norm{\tilde w_k\zeta_k}_{L^{2^\ast_s}(B_k)} \ge C_5  h ^{\alpha} 2^{-\alpha k} Y_{k+1}^{\frac{1}{1+\beta}}
\]
where also the constant $C_5$ depends just on $N$ and $s$.
To see that, we use that
for all points $x\in B_{k+1}$ we have
\(
\tilde w_k(x)\zeta_k(x)=\tilde w_k(x)
= 2^{-(k+2)}h+w_{k+1}(x),
\)
whence it follows that
$(\tilde w_k\zeta_k)^{2^{\ast}_s}\ge ( h/4)^{2^\ast_s-2} 2^{-(2^\ast_s-2)k}w_{k+1}^2 $  in $B_{k+1}$, and this gives the desired lower bound.

Therefore, for appropriate constants $C_6$ and $\Lambda_0$, depending only on $N$ and $s$, we have
\[
Y_{k+1}\le \frac{C_6\Lambda_0^k Y_k^{1+\beta}}{ r^{\frac{N}{2}(1+\beta)}h^{\alpha(1+\beta)}}
\quadra*{
1+\tonda*{\frac{\Tail(w,x_0,r/2)}{h}}^\frac12 
+\tonda*{\frac{r^{2s}\norm f_{L^\infty(\mathscr U)}}{h}}^\frac12}^{1+\beta}
\]
and the latter takes the form
$Y_{k+1} \le r^{-\frac{N}{2}(1+\beta)}h^{-\alpha(1+\beta)}\delta^{-\frac{1+\beta}{2}} C_7\Lambda_0^k Y_k^{1+\beta}$,
\begin{subequations}
provided that
\label{htildelarge}
\begin{equation}
 h\ge \delta\Tail(w,x_0,r/2) + \delta r^{2s}\norm f_{L^\infty(\mathscr U)}
\end{equation}
Then, by setting $C_8=C_7^{1/\beta}$,
$\Lambda = \Lambda_0^{1/\beta}$ and $Z_k = C_8\Lambda^k Y_k$, we obtain the recursive relation
\[
	Z_{k+1}\le \tonda*{\Lambda r^{-\frac{N}{2}(1+\beta)}h^{-\alpha(1+\beta)}\delta^{-\frac{1+\beta}{2}} Z_k^{\beta}}Z_k
\]
If we also have
\begin{equation}
 h\ge \delta^{-\frac{1}{2\alpha}} (C_7\Lambda)^{\frac{1}{\alpha(1+\beta)}} 
 r^{-\frac{N}{2\alpha}}\norm w_{L^2(B_r(x_0))}
\end{equation}
then, from the recursive relation, we infer by induction that $Z_k\le Z_0$ for all $k\in\mathbb N$, which means, by
construction, that $Y_k\le \Lambda^{-k}Y_0$. In view of \eqref{T666}, it follows that
\[
\int_{B_{r/2}(x_0)} (w- h)_+^2\,dx
\le \liminf_{k\to\infty}	
\int_{B_k}w_k^2\,dx\le\lim_{k\to\infty}  \Lambda^{-2k} \int_{B_r(x_0)}w^2\,dx=0
\]
\end{subequations}
For every $h>0$, the procedure can be repeated for all those $\tilde h$ that
meet the requirement that both the lower bounds in \eqref{htildelarge} for $h$ hold, leading one to the conclusion that
\[
w\le 
\delta\Tail(w,x_0,r/2)  + 
\delta r^{2s}\norm f_{L^\infty(\mathscr U)}
+ C_8 \delta^{-\frac{1}{2\alpha}} r^{-\frac{N}{2\alpha}} \norm w_{L^2(B_r(x_0))}
\qquad \text{a.e.\ in $B_{r/2}(x_0)$}
\]
where $C_8$ depends only on $N$ and $s$. Since $\alpha=2s/N$, that ends the proof.
\end{proof}

\section{Functional inequalities with special singular weights}\label{techn2}

In the present section, we introduce a couple of Hardy-type inequalities.
In the following proposition, we see that Lane-Emden inequalities (see \cite[Section 3]{BFR}) are valid also in the non-local case.

\begin{proposition}\label{prop:LEineq}
Let $q\in(1,2)$ and $u\in C^\infty_0(\Omega)$. Then
\begin{equation}
\label{LEineq}
	\int_{\Omega} \frac{u^2}{w_{\Omega,s,q}^{2-q}}\,dx\le \int_{\R^{N}}\int_{\R^N} \frac{(u(x)-u(y))^2}{\abs{x-y}^{N+2s}}\,dx\,dy
\end{equation}
with the agreement that the left integrand be $0$ at all points where $w_{\Omega,s,q}=+\infty$.
\end{proposition}

\begin{proof}
We first prove \eqref{LEineq} in the special case of a bounded open set.
We write $w=w_{\Omega,s,q}$, and we take $\varepsilon>0$. By \cref{linfty}, $w\in \X\cap L^\infty(\Omega)$. Hence, so does $(w+\varepsilon)^{-1}$,
because $t\mapsto (t+\varepsilon)^{-1} $ is a Lipschitz function on $(0,\infty)$. Then, by \cite[Lemma 2.4]{BC} we can
plug $\varphi=u^2/(w+\varepsilon)$ into the equation for $w$ and get
\[
\int_\Omega w(x)^{q-1} \frac{u(x)^2}{w(x)+\varepsilon}\,dx = 	\int_{\R^{N}}\int_{\R^N} \frac{w(x)-w(y)
}{\abs{x-y}^{N+2s}}\tonda*{\frac{u(x)^2}{w(x)+\varepsilon}-\frac{u(y)^2}{w(y)+\varepsilon}
}\,dx\,dy
\]
for all $\varepsilon>0$. In view of \cite[Proposition 4.2]{BF14}, and recalling that $w>0$ a.e.\ in $\Omega$, by Fatou's lemma it follows that 
\[
\int_\Omega \frac{u^2}{w^{2-q}}\,dx \le \int_{\R^{N}}\int_{\R^N} \frac{(u(x)-u(y))^2}{\abs{x-y}^{N+2s}}\,dx\,dy
\]

For the general case, we take $R>0$ so large that the support of $u$ is contained in $B_r$ for all $r\ge R$. For all such radii $r$,
by the material above we have
\[
	\int_{\Omega\cap B_r} \frac{u^2}{w_r^{2-q}}\,dx \le \int_{\R^{N}}\int_{\R^N} \frac{(u(x)-u(y))^2}{\abs{x-y}^{N+2s}}\,dx\,dy
\]
where $w_r $ is the $(s,q)$--Lane-Emden density of $\Omega\cap B_r$.
In view of \cref{wqs}, by Fatou's lemma we get the conclusion passing to the limit as $r\to\infty$.
\end{proof}

The more familiar Hardy-type inequality of next proposition implies
some restriction on $\Omega$. The assumption made below is not optimal, though; for instance, a uniform exterior cone condition is also a valid assumption. More generally, for the statement to hold true it would be sufficient that no boundary point belong to the measure-theoretic interior of $\Omega$ (see \cite{CC}).

\begin{proposition}\label{Hardyneq}
Let $s\in(0,1)$ and let $\Omega\subset\R^N$ be an open bounded Lipschitz set. Then, for all $u\in C^\infty_0(\Omega)$,
\begin{equation}
\label{Hardyneq-EQ}
\int_{\Omega}\frac{u(x)^2}{\dist(x,\partial\Omega)^{2s}}\,dx\le C
\int_{\mathbb R^N}\int_{\mathbb R^N} \frac{(u(x)-u(y))^2}{\abs{x-y}^{N+2s}}\,dx\,dy
\end{equation}
for a constant $C>0$ depending only on $\Omega$.
\end{proposition}

Before proving \cref{Hardyneq}, we make a brief comment on
\eqref{Hardyneq-EQ}. When it comes to fractional Hardy inequalities, there are a number of variants of the same statement. 
A stronger one just involves the Sobolev-Slobodeckij seminorm $[u]_{H^s(\Omega)}$ in the right hand side (instead of taking integrals on the whole of $\mathbb R^{N}$), implying various restrictions both on $\Omega$ and on $s$: for a more detailed account on the topic, we refer to \cite{CC,D,DV0,DV,Sk}. 
Here, incidentally, in view of \cref{rmk-analog} we may point out the following.

\begin{corollary}
If $s\in(0,1)$ and $2s\neq N$, then, under the assumptions
of \cref{Hardyneq},
\begin{equation*}
\int_{\Omega}\frac{u(x)^2}{\dist(x,\partial\Omega)^{2s}}\,dx\le C
\tonda*{
\int_{\Omega}\int_{\Omega} \frac{(u(x)-u(y))^2}{\abs{x-y}^{N+2s}}\,dx\,dy
+\int_\Omega u^2\,dx}
\end{equation*}
for all $u\in C^\infty_0(\Omega)$.
\end{corollary}

\begin{proof}[Proof of \cref{Hardyneq}]
By assumption, $\Omega$ satisfies the uniform exterior cone condition, i.e., that there exists $\ell>0$ and a cone $K$, with given aperture, such that every boundary point $\xi$ is the vertex 
of a cone $K_\xi$ isometric to $K$ that satisfies $K_\xi\cap B_\ell(\xi)\subset \mathbb R^N\setminus\Omega$.

For ease of notation, we write $\delta(x)\coloneqq\dist(x,\partial\Omega)$.
For all $x\in\Omega$ with $\delta(x)\ge \ell$, we can pick $\xi_x\in\partial\Omega$ with minimum distance to $x$ and we have
$\abs{x-y}\le \abs{\xi_x-y} +\delta(x)\le 2\delta(x)$ for all $y\in K_{\xi_x}\cap B_\ell(\xi_x)$,
whence it follows that
\[
	\int_{K_{\xi_x}\cap B_\ell(\xi_x)} \frac{dy}{\abs{x-y}^{N+2s}}
	\ge (2\delta(x))^{-(N+2s)}
	\abs{K_{\xi_x}\cap B_\ell(\xi_x)}
    	\ge \frac{\theta \ell^N\delta(x)^{-2s}}{2^{N+2s}D^N N}
\]
where $\theta = \mathscr H^{N-1}(K\cap \partial B_1(0))$ and $D$ is the diameter of $\Omega$.

The inequality 
$\abs{x-y}\le \abs{\xi_x-y} +\delta(x)$ holds also for all points $x\in\Omega$ with $\delta(x)\le \ell$, and we infer that
\[
\int_{K_{\xi_x}\cap B_\ell(\xi_x)} \frac{dy}{\abs{x-y}^{N+2s}}\ge 
\int_0^{\delta(x)}
\frac{\theta \rho^{N-1}\,d\rho}{(\rho+\delta)^{N+2s}} = 
\frac{\theta}{N\delta(x)^{2s}}\int_0^1\frac{dt}{(1+t^{1/N})^{N+2s}}
\ge \frac{\theta\delta(x)^{-2s}}{2^{N+2s}N}
\]

Since for all $x\in\Omega$ we have 
$K_{\xi_x}\cap B_\ell(\xi_x)\subset \mathbb R^N\setminus\Omega$, 
it follows that
\[
	\int_{\mathbb R^N\setminus\Omega} \frac{dy}{\abs{x-y}^{N+2s}}
	\ge \frac{\theta\delta(x)^{-2s}}{2^{N+2s}N}\tonda*{\frac{\ell^N}{D^N}\wedge1}
\]
That gives the desired conclusion, because
for all $u\in C_0^\infty(\Omega)$ we have
\[
\int_{\mathbb R^N}\int_{\mathbb R^N}\frac{(u(x)-u(y))^2}{\abs{x-y}^{N+2s}}
=\int_\Omega\int_\Omega
\frac{(u(x)-u(y))^2}{\abs{x-y}^{N+2s}}
+2 \int_{\Omega} u(x)^2\int_{\mathbb R^N\setminus \Omega} \frac{dy}{\abs{x-y}^{N+2s}} \qedhere
\]
\end{proof}

\begin{remark}
For the use we shall make of \cref{Hardyneq}, we don't need 
to pay much attention to the explicit value of the constant $C>0$. For sure, the proof presented implies a very rough estimate of the optimal (unknown) constant.
\end{remark}

\begin{remark}
By density, the inequality holds for all functions that belong to $\mathcal{D}_0^{s,2}(\Omega)$. Given $p\in(1,\infty)$, a similar inequality, with suitable adjustments to the exponents, is valid
for functions in the homogeneous fractional Sobolev space
$\mathcal{D}_0^{s,p}(\Omega)$ defined as the completion of $C^\infty_0(\Omega)$ with respect to
\[
	\tonda*{\int_{\mathbb R^N}\int_{\mathbb R^N} \frac{|u(x)-u(y)|^p}{\abs{x-y}^{N+sp}}\,dx\,dy}^\frac1p
\]
and this can be seen by minor changes in the proof presented here. This variant was considered, for example, in \cite{BC}, where the authors provide a constant that works on convex open sets, with stable asymptotic
behaviour as $s\nearrow1$.
\end{remark}

\section{Universal bounds for Lane-Emden densities of unbounded open sets}

The following is the non-local counterpart of \cite[Proposition 4.3]{BFR}.

\begin{proposition}\label{linftty2}
Let $q\in(1,2)$ and assume that $\lambda_1(\Omega,s,2)>0$.
Then, $w_{\Omega,s,q}\in L^\infty(\Omega)$ and
there exists a constant $C$, depending only on $N$, $s$ and $q$, such that 
\begin{equation}
\label{estw}
 \norm{w_{\Omega,s,q}}_{L^\infty(\Omega)}^{2-q}\le C\lambda_1(\Omega,s,2)^{-1}
\end{equation}
Conversely, for all $q\in(1,2)$, if $w_{\Omega,s,q}\in L^\infty(\Omega)$, then  $\lambda_1(\Omega,s,2)\ge \norm{w_{\Omega,s,q}}_{L^\infty(\Omega)}^{q-2} $.
\end{proposition}
\begin{proof}
Let us write $w=w_{\Omega,s,q}$. The last statement is a consequence of  \cref{prop:LEineq}. Then, we assume that
$\lambda_1(\Omega,s,2)>0$
and we prove the following fact: there exists a constant $C_2(N,s)$, that
only depends on $N$ and $s$, such that
\begin{equation}
\label{retro0}
	\norm w_{L^\infty(\Omega)}^{2-q}\lambda_1(\Omega,s,2)\le C_1(N,s,q)
\end{equation}
 holds with a suitable constant $C_1(N,s,q)$, depending only on $N$, $s$ and $q$,
provided that
\begin{equation}\label{retro2}
	\norm w_{L^\infty(\Omega)}^{2-q}\ge C_2(N,s)
\end{equation}
That fact would imply 
\[
 \norm{w_{\Omega,s,q}}_{L^\infty(\Omega)}^{2-q}\lambda_1(\Omega,s,2)
\le \max\{C_2\lambda_1(\Omega,s,2), C_1\}
 \]
whence we would infer \eqref{estw} by a scaling argument, because for all $t>0$
we have
\begin{align*}
\norm{w_{t\Omega,s,q}}_{L^\infty(t\Omega)}^{2-q}
&=t^{2s} \norm{w_{\Omega,s,q}}_{L^\infty(\Omega)}^{2-q}\\
\lambda_1(t\Omega,s,2)
&=t^{-2s}\lambda_1(\Omega,s,2)
\end{align*}

In order to prove that \eqref{retro2} implies \eqref{retro0}, as desired,
for appropriate choices of constants,
we follow the lines of the proof of \cite[Theorem 9]{BB}. 
Since the $L^\infty$-norm
is lower semicontinuous with respect to the pointwise (monotone)
convergence and the first eigenvalue $\lambda_1(\cdot,s,2)$
is monotone non-increasing with respect to set inclusion, in order to prove the claim we may assume $\Omega$ to be smooth and bounded, up to an approximation argument. So, by arguing under this assumption, in view of \cref{linfty} we will assume
$w$ to belong to $ L^\infty(\Omega)$ and to achieve its maximum
at an interior point, that we may consider to be the origin in $\R^N$
up to an unessential translation.

We now identify $w$ with the function that agrees with $w$ in $\Omega$ and equals zero everywhere else and we claim that
$w$ is a weak subsolution of the fractional Lane-Emden equation \eqref{LE} in $\R^N$. To see this\footnote{We owe the approximation trick used in the proof of this claim to a gentle advice by Lorenzo Brasco.}, we fix a non-negative function $\eta\in C^\infty_0\tonda*{\mathbb R^N}$ and, for every $\varepsilon>0$, we take a monotone non-decreasing Lipschitz continuous function
$H_\varepsilon\colon\mathbb R\to\mathbb R$, with $H_\varepsilon(u)=0$ for all $u\le0$ and $H_\varepsilon(u)=1$ for all $u\ge\varepsilon$. Then
 \begin{equation}\label{ssol:1}
  \int_{\R^N}\int_{\R^N} \frac{w(x)-w(y)}{\abs{x-y}^{N+2s}}[ H_\varepsilon(w(x))\eta(x)-H_\varepsilon(w(y))\eta(y)]\,dx\,dy
 = \int_\Omega w^{q-1}H_\varepsilon(w)\eta\,dx
 \end{equation}
 because of the weak equation for $w$ with $H_\varepsilon(w)\eta$ as a test function.
To handle the left hand side of \eqref{ssol:1}, 
we write the identity
\(
 2(a\xi-b\zeta) 
 = (a+b)(\xi-\zeta)
 +(a-b)(\xi+\zeta) 
\)
with $a=\eta(x)$, $b=\eta(y)$, $\xi=H_\varepsilon(w(x))$ and $\zeta=H_\varepsilon(w(y))$. After multiplying
the result by $w(x)-w(y)$ and integrating
against the singular kernel on $\Omega\times\Omega$, we see that
\begin{equation}
\label{ssol:4}
\begin{split}
2 & \int_\Omega\int_\Omega
\frac{w(x)-w(y)}{\abs{x-y}^{N+2s}}
[H_\varepsilon(w(x))\eta(x)-H_\varepsilon(w(y))\eta(y)]\,dx\,dy\\
& = 
\int_\Omega\int_\Omega
\frac{w(x)-w(y)}{\abs{x-y}^{N+2s}}
(\eta(x)+\eta(y))
[H_\varepsilon(w(x))-H_\varepsilon(w(y))]\,dx\,dy\\
& + 
\int_\Omega\int_\Omega
\frac{w(x)-w(y)}{\abs{x-y}^{N+2s}}
(\eta(x)-\eta(y))
[H_\varepsilon(w(x))+H_\varepsilon(w(y))]\,dx\,dy
\end{split}
\end{equation}
Notice that the first integral in the right hand side of \eqref{ssol:4} is non-negative,
due to the monotonicity of the function $H_\varepsilon$. Thus
\begin{equation}
\label{ssol:5}
\begin{split}
&\int_{\R^N}\int_{\R^N}
\frac{w(x)-w(y)}{\abs{x-y}^{N+2s}}
[H_\varepsilon(w(x))\eta(x)-H_\varepsilon(w(y))\eta(y)]\,dx\,dy
\\ & \ge
\frac12\int_\Omega\int_\Omega
\frac{w(x)-w(y)}{\abs{x-y}^{N+2s}}
(\eta(x)-\eta(y))
[H_\varepsilon(w(x))+H_\varepsilon(w(y))]\,dx\,dy\\
& +
2\int_\Omega\int_{\R^N\setminus\Omega}
\frac{w(x)}{\abs{x-y}^{N+2s}}
[H_\varepsilon(w(x))\eta(x)-H_\varepsilon(w(y))\eta(y)]\,dx\,dy
\end{split}
\end{equation}
By dominated convergence theorem, the limit as $\varepsilon\to0^+$
in \eqref{ssol:1} and \eqref{ssol:5} gives
\[
\int_\Omega\int_\Omega
\frac{w(x)-w(y)}{\abs{x-y}^{N+2s}}
(\eta(x)-\eta(y))\,dx\,dy
+2\int_\Omega\int_{\R^N\setminus\Omega}
\frac{w(x)\eta(x)}{\abs{x-y}^{N+2s}}\,dx\,dy
\le \int_\Omega w^{q-1}\eta\,dx
\]
and that proves the claim.

We let $r$ be a positive radius, that will be chosen later, and we take a cut-off
$\zeta\in C^\infty_0(\Omega)$ from the ball $B_{r/2}$ to $B_r$, with $\abs{\nabla\zeta}\le \frac2r$. Since $w$ is a weak subsolution of \eqref{LE},
the localised Caccioppoli estimate of \cite[Proposition 3.5]{BP}, with $F=w^{q-1}$, 
$p=2$, $\beta=1$, $\delta=0$, $L=1$ and $\Omega'=B_{r}$, gives
\begin{equation}
\label{estw1}
	\int_{B_r}\int_{B_r} \frac{(w(x)\zeta(x)-w(y)\zeta(y))^2}{\abs{x-y}^{N+2s}}\,dx\,dy\le C_3(N,s) \tonda*{w(0)^q r^N +w(0)^2r^{N-2s}}
\end{equation}
where $C_3(N,s)>0$ depends only on $N$ and $s$. Moreover, by the fact that $w\in L^\infty(\Omega)$,
\begin{equation}
\label{estw2}
	\int_{B_r}\int_{\R^N\setminus B_r} \frac{(w(x)\zeta(x)-w(y)\zeta(y))^2}{\abs{x-y}^{N+2s}}\,dx\,dy\le C_4(N,s) w(0)^2r^{N-2s}
\end{equation}
where $C_4(N,s)>0$ depends only on $N$ and $s$. Also, by \cref{lm:subsoLE} we have
\[
	\int_{B_r} w^2\,dx \ge C_5(N,s,q)r^N\tonda*{w(0)
	-\delta\Tail(w,0,r/2) -\delta r^{2s} w(0)}^2
\]
where the constant $C_5(N,s,q)$ depends only on $N$, $s$ and $q$
and $\delta\in(0,1]$ is a parameter that we can take as small as we wish.
By combining the latter with \eqref{estw1} and \eqref{estw2}, for
$\delta$ smaller than an appropriate $\delta_0(N,s)\in(0,1]$, we obtain
\begin{equation}
\label{last}
\lambda_1(\Omega,s,2) \le C_6(N,s,q) \tonda*{w(0)/2-
\delta r^{2s} w(0)^{q-2}}^{-2}\tonda*{w(0)^q + w(0)^2r^{-2s}}
\end{equation}
where we set $C_6 = 2(C_3+C_4) C_5^{-1}$ and we used 
the fact that the function $w\zeta$ is an admissible competitor for the infimum that
defines the constant $\lambda_1(\Omega,s,2)$. Then, we take
$\delta\le 2^{-q}\wedge \delta_0(N,s)$. Hence, with the choice
\[
	r = \tonda*{\tfrac12 w(0)}^\frac{2-q}{2s}
\]
we have
$w(0)-
\delta r^{2s} w(0)^{q-2}\ge w(0)/4$, and \eqref{last} gives \eqref{retro0} with $C_1 = 16\tonda*{1+2^{2-q}}C_6$.
\end{proof}

Under the stronger assumption that $\lambda_1(\Omega,s,q)>0$, we have the following estimate.

\begin{proposition}\label{linftyw}
Let $q\in(1,2)$ and $\lambda_1(\Omega,s,q)>0$. Then, $w_{\Omega,s,q}\in L^\infty(\Omega)$ and
\[
	\norm{w_{\Omega,s,q}}_{L^\infty(\Omega)}\le \mathcal{C} \lambda_1(\Omega,s,q)^{-\gamma}
\]
where the constant $\mathcal{C}>0$ and the exponent $\gamma>0$ depends only on $N$, $s$ and $q$.
\end{proposition}

\begin{proof}
We note that $w_{\Omega,s,q}$ is the first $q$-semilinear $s$-eigenfunction
with $L^q(\Omega)$-norm $\lambda_1(\Omega,s,q)^{\frac{1}{q-2}}$. Then, the estimate follows at once by \cref{linfty}.
\end{proof}

\begin{remark}
We notice that \cref{linftyw} can also be seen as a particular case of the general estimate \eqref{estw} of \cref{linftty2}.
Indeed, the positivity of the greatest lower bound $\lambda_1(\Omega,s,2)$ 
for the spectrum of the fractional (linear) $s$-Laplacian is, by definition,
equivalent to the continuity of the embedding $\X\hookrightarrow L^2(\Omega)$. Domains with this property
are not necessarily bounded, nor are they required to have finite measure;
also, an open set $\Omega$ may support a Sobolev-Poincaré inequality that makes
$\lambda_1(\Omega,s,2)$ strictly positive even if $\lambda_1(\Omega,s,q)=0$ for all $q\in(1,2)$
(examples are provided by domains of the form $\omega\times(-M,M)$, with $M>0$ and $\omega$ bounded in $\R^{N-1}$).
Conversely, given any $q\in(1,2)$, the fact that $\lambda_1(\Omega,s,q)>0$ implies that
$\lambda_1(\Omega,s,2)>0$, too; in fact, it implies that the embedding $\X\hookrightarrow L^2(\Omega)$ is compact, by interpolation (see \cite[Lemma~2.3]{F19}).
\end{remark}

\section{Local in \texorpdfstring{$L^1$}{L1} uniqueness for fractional Lane-Emden positive solutions}\label{sec:isolation}
The following proposition is the non-local counterpart of \cite[Proposition 4.1]{BDF}.

\begin{proposition}\label{abstractprop}
Let $q\in(1,2)$ and assume that the weighted space
\begin{equation}
\label{wesp}
	L^2(\Omega,w_{\Omega,s,q}^{q-2}) =
	\graffa*{u\in L^1_{\rm loc}(\Omega):\int_\Omega w_{\Omega,s,q}^{q-2}u^2\,dx<+\infty}
\end{equation}
contains $\mathcal{D}_0^{s,2}(\Omega)$ with compact embedding. Then, every
critical point of
\[
	\frac12\int_{\mathbb R^N}\int_{\mathbb R^N} \frac{(u(x)-u(y))^2}{|x-y|^{N+2s}}\,dx\,dy
	-\frac1q \int_\Omega{\abs u^q\,dx}
\]
must satisfy $\norm{u-w_{\Omega,s,q}}_{L^1(\Omega)}\ge\delta$, where
$\delta>0$ depends only on $s$, $q$, $\Omega$ and $N$.
\end{proposition}

\begin{proof}
We will prove a contrapositive statement: if a sequence $(u_n)_{n\in\mathbb N} $, consisting of weak solutions of the fractional Lane-Emden equation \eqref{LE}, converges to $w\coloneqq w_{\Omega,s,q}$ in $L^1(\Omega)$, then
\begin{equation}
\label{contrapst}
\int_{\mathbb R^N}\int_{\mathbb R^N} \frac{(\psi(x)-\psi(y))^2}{\abs{x-y}^{N+2s}}\,dx\,dy
\le (q-1) \int_\Omega w^{q-2}\psi^2\,dx
\end{equation}
Note that \eqref{contrapst} is in contradiction with \cref{prop:LEineq}, because $1<q<2$.

By setting $Q_n= \left(w^{q-1}-\abs{u_n}^{q-2}u_n\right)/(w-u_n)$
at all points where $w\neq u_n$ and $Q_n=0$ elsewhere, the weak
equation for the difference $w-u_n$ takes the form
\[
\int_{\mathbb R^N}\int_{\mathbb R^N} \frac{((w-u_n)(x)-(w-u_n)(y))(\varphi(x)-\varphi(y))}{\abs{x-y}^{N+2s}}\,dx\,dy
= \int_\Omega Q_n (w-u_n)\varphi\,dx
\]
for $\varphi\in\X$; with the choice $\varphi= t_n^{-1}(w-u_n)$, where 
\[
t_n=\int_\Omega w^{q-2}(w-u_n)^2\,dx
\]
it follows that
\begin{equation}
\label{phipsi}
\frac{1}{t_n}\int_{\mathbb R^N}\int_{\mathbb R^N} \frac{((w-u_n)(x)-(w-u_n)(y))^2}{\abs{x-y}^{N+2s}}\,dx\,dy=
\int_\Omega Q_n\tonda*{\frac{w-u_n}{\sqrt{t_n}}}^2\,dx
\end{equation}
By \cite[Lemma A.1]{BDF}, we have the following pointwise bound 
\begin{equation}
\label{ptwbdd}
0\le Q_n(x)\le 2^{2-q} w^{q-2}(x) \qquad \text{for all $x\in\Omega$}
\end{equation} 
and, by construction, that prevents
the right integral in \eqref{phipsi} from exceeding the constant $2^{2-q}$.
Therefore, setting $\psi_n=(w-u_n)/\sqrt{t_n}$ defines a bounded sequence in $\X$, which clearly has unit norm in the weighted space \eqref{wesp}. 

By assumption, we deduce that $\psi_n$ converges weakly in $\X$ and
strongly in the weighted space \eqref{wesp} to a non-zero limit $\psi$. Thus, by \eqref{phipsi}, we can write
\begin{equation}
\label{phipsi2}
\int_{\mathbb R^N}\int_{\mathbb R^N} \frac{(\psi_n(x)-\psi_n(y))^2}{\abs{x-y}^{N+2s}}\,dx\,dy
= \int_\Omega Q_n \tonda*{\psi_n^2-\psi^2}\,dx
+ \int_\Omega Q_n \psi^2\,dx
\end{equation}
The convergence of the sequence $\psi_n$ implies
\begin{equation}
\label{limsup1}
	\limsup_{n\to\infty}\int_\Omega Q_n \tonda*{\psi_n^2-\psi^2}\,dx\le0
\end{equation}
because, by the pointwise bound \eqref{ptwbdd} and by H\"older's inequality, we have
\begin{multline*}
	\int_\Omega Q_n \tonda*{\psi_n^2-\psi^2}\,dx
	\le	2^{2-q}  \tonda*{\int_\Omega w^{q-2} (\psi_n-\psi)^2\,dx}^\frac12\times\\ 
	\times\quadra*{\tonda*{\int_\Omega w^{q-2} \psi_n^2\,dx}^\frac12+
	\tonda*{\int_\Omega w^{q-2} \psi^2\,dx}^\frac12}
\end{multline*}

In order to deal with the second integral in the right hand side of \eqref{phipsi2},
we would better handle the pointwise limit behaviour of $Q_n$. Since
\[
w^{q-1}-\abs{u_n}^{q-2}u_n=-\int_0^1\diff*{\quadra*{\abs{w+t(u_n-w)} ^{q-2}(w+t(u_n-w))}}t\,dt
\]
for every $x\in\Omega$, we have
\begin{equation}
\label{prefatou2}
	Q_n(x) \le (q-1) \int_0^1f_n(x,t)\,dt\qquad
	\text{where $f_n(x,t) = \abs{(1-t)w(x)+tu_n(x)}^{q-2}$}
\end{equation}
The H\"older continuity of $\tau\mapsto \tau^{2-q}$
and the convexity of $\tau\mapsto\tau^{q-2}$ imply 
\[
\abs*{|a+t(b-a)|^{q-2} -a^{q-2}}
\le \frac{t^{2-q}}{a^{2-q}} (b-a)^{2-q} \quadra*{
(1-t) a^{q-2}+tb^{q-2}}
\]
for  all $t\in[0,1]$ and for all $a,b>0$.
Then, at all points $x$ where $u_n(x)>0$, we have
\begin{equation}
\label{prefatou3}
	\sup_{t\in[0,1]} \abs*{f_n(x,t)-w^{2-q}(x)}
	\le 
	\frac{\abs{w(x)-u_n(x)}^{2-q}}{w(x)^{2-q}} 
	\quadra*{(1-t)w(x)^{q-2}+tu_n(x)^{q-2}}
\end{equation}
As $u_n$ converges to $w$ in $L^1(\Omega)$, a subsequence (not relabelled) also converges  pointwise a.e.\ in $\Omega$. In view of \eqref{prefatou3}, that assures the uniform convergence of $f_n(x,\cdot)$ to
the constant $w(x)^{2-q}$ for all $x$ out of a negligible set, so that
\begin{equation}
\label{prefatou4}
\lim_{n\to\infty} \int_0^1f_n(x,t)\,dt = w(x)^{2-q}\qquad \text{for a.e.\ $x\in\Omega$}
\end{equation}
From \eqref{prefatou2} and \eqref{prefatou4} we infer that
\begin{equation*}
	\limsup_{n\to\infty}Q_n\psi^2 \le (q-1) w^{2-q}\psi^2\qquad
	\text{a.e.\ in $\Omega$}
\end{equation*}
From this and from \eqref{ptwbdd}, by reverse Fatou's lemma, we deduce that
\begin{equation}
\label{limsup2}
	\limsup_{n\to\infty}\int_\Omega Q_n\psi^2\,dx\le (q-1)\int_\Omega w^{2-q} \psi^2\,dx
\end{equation}
Inserting \eqref{limsup1} and \eqref{limsup2}
in the identity
\eqref{phipsi2} and using the lower semicontinuity of the left hand side
of \eqref{phipsi2} with respect to the weak convergence in $\mathcal{D}_0^{s,2}(\Omega)$, we arrive at \eqref{contrapst}, as desired.
\end{proof}

\begin{remark}
In view of \cref{prop:LEineq}, the embedding $\X\hookrightarrow L^2(\Omega,w_{\Omega,s,q}^{q-2})$ is continuous, for example, on all open sets with finite volume. The stronger requirement that it be compact may be met
under higher regularity assumptions on $\partial\Omega$.
\end{remark}

\begin{lemma}\label{LEL1uniq}
Let $q\in(1,2)$, let $\Omega\subset\R^N$ be a bounded open set with $C^{1,1}$ boundary
and let $v\in C^\infty_0(\Omega)$. Then
\begin{equation}
\label{holderrr}
\int_\Omega w_{\Omega,s,q}^{q-2}v^2\,dx\le \tonda*{\int_{\mathbb R^N}\int_{\mathbb R^N} 
\frac{(v(x)-v(y))^2}{\abs{x-y}^{N+2s}}\,dx\,dy}^\frac{2-q}{2}
\norm v_{L^2(\Omega)}^q
\end{equation}
\end{lemma}

\begin{proof}
By Hopf's lemma for the fractional Laplacian (see \cite[Lemma 7.3]{R2}) we have a constant $C>0$,
only depending on $\Omega$, $N$, $q$ and $s$, such that\footnote{The more
precise asymptotic boundary behaviour
$w_{\Omega,s,q}\asymp\dist(\cdot,\partial\Omega)^{s}$
is known: for the semilinear equation we refer to Theorem 6.4 and the following remarks in \cite{BFV} (alternatively, see \cite{RS} for the linear equation with a bounded right hand side, which is also relevant to our case thanks to \cref{linfty}).}
\begin{equation}
\label{rosoton}
	w_{\Omega,s,q}(x) \ge C\dist(x,\partial\Omega)^{s}
\end{equation}
Since $q\in(1,2)$, by H\"older's inequality with exponents $\frac2{2-q}$ and $\frac2q$ we have
\begin{equation}
\label{holderr}
	\int_\Omega \dist(x,\partial\Omega)^{s(q-2)}v^2\,dx
	\le\tonda*{\int_\Omega \frac{v^2}{\dist(x,\partial\Omega)^{2s}}\,dx}^{\frac{2-q}{2}}\tonda*{\int_\Omega v^2\,dx}^\frac{q}{2}
\end{equation}
Then, by \eqref{rosoton}, \cref{Hardyneq} and \eqref{holderr}, we improve the fractional
Lane-Emden inequality \eqref{LEineq} to \eqref{holderrr}.
\end{proof}

The conclusion of the previous lemma assures compactness for the weighted embedding. Thus, we end this section with the remark that the isolation of fractional Lane-Emden densities holds, for example, on open sets with smooth boundary.

\begin{proposition}\label{isolationC11}
Let $q\in(1,2)$ and let $\Omega\subset\R^N$
be a bounded open set with $C^{1,1}$ boundary. Then, the conclusion of \cref{abstractprop} holds.
\end{proposition}

\begin{proof}
By assumption, $\mathcal{D}_0^{s,2}(\Omega)\hookrightarrow L^2(\Omega)$ is compact; this and \cref{LEL1uniq} imply the compactness of the embedding $\mathcal{D}_0^{s,2}(\Omega)\hookrightarrow L^2(\Omega,w_{\Omega,s,q}^{q-2})$, too.
\end{proof}

\section{Proof of the main results}\label{sec:mainproof}
\subsection{Proof of Theorem \ref{teoa}}
Because $q\in(1,2)$, the assumption $\lambda_1(\Omega,s,q)>0$ implies the compactness
of the embedding $\X\hookrightarrow L^q(\Omega)$ (see \cite[Theorem 1.3]{F19}).
Then, a first eigenfunction exists by \cref{prop:exist1}.
Also, \cref{prop:simpl} entails
uniqueness up to proportionality, and
the last statement is true by \cref{prop:sign2}.\qed

\subsection{Proof of Theorem \ref{teo-isol}}
Arguing by contradiction, we assume that a sequence
$(\lambda_n)_{n\in\mathbb N}\subset \mathfrak{S}(\Omega,s,q)$ 
converges to $\lambda_1(\Omega,s,q)$. For each $\lambda_n$, we pick
an eigenfunction $u_n$
with
unit norm in $L^q(\Omega)$. That defines 
a bounded sequence
in $\mathcal{D}_0^{s,2}(\Omega)$, due to equation~\eqref{eqw} with $\lambda=\lambda_n$
and $u=\varphi=u_n$. Then, by possibly passing to a subsequence,
we may assume that $u_n$ converges
weakly in $\mathcal{D}_0^{s,2}(\Omega)$ and strongly in $L^q(\Omega)$
to a limit function $u$ with unit norm in $L^q(\Omega)$.
Hence, by passing to the limit as $n\to\infty$ in \eqref{eqw} with $u=u_n $ and $\lambda=\lambda_n$, it is easily seen that 
$u$ is a first $q$-semilinear
$s$-eigenfunction. Owing to \cref{teoa}, up to changing everywhere sign to each element of the sequence, $u>0$ and its multiple $w=\lambda_1(\Omega,s,q)^{\frac1{q-2}} u$ is the fractional Lane-Emden density of $\Omega$.
Moreover, each function $v_n=\lambda_1(\Omega,s,q)^{\frac1{q-2}}u_n$ is a weak solution of the fractional Lane-Emden equation \eqref{LE}; yet, by construction,
$v_n$ converges to $w_{\Omega,s,q}$ in $\mathcal{D}_0^{s,2}(\Omega)$, in contradiction with \cref{isolationC11}.
\qed

\appendix

\section{Strong minimum principle}
The following lemma is an immediate consequence of
inequality $(\abs a-\abs b)^2\le (a-b)^2$, that is strict if and only if $ab<0$.
\begin{lemma}\label{lm:sign1}
For all $u\in\X$, we have
\begin{equation}
\label{passab}
\int_{\R^{N}}\int_{\R^N}\frac{(\abs{u(x)}-\abs{u(y)})^2}{\abs{x-y}^{N+2s}}\,dx\,dy\le
		\int_{\R^{N}}\int_{\R^N}\frac{(u(x)-u(y))^2}{\abs{x-y}^{N+2s}}\,dx\,dy	
\end{equation}
with strict inequality unless either $u\ge0$ or $u\le0$ a.e.\ in $\Omega$.
\end{lemma}

The following form of the minimum principle for weak supersolutions is well known. We present the proof 
for convenience of the reader and we point out that $\Omega$ is not required to be connected.
\begin{proposition}\label{nonlocmaxprinc}
Let $u \in\X$ 
satisfy
\[
	\int_{\R^{N}}\int_{\R^N}\frac{(u(x)-u(y))(\varphi(x)-\varphi(y))}{\abs{x-y}^{N+2s}}\,dx\,dy \ge 0
	\qquad \text{for all non-negative $\varphi\in C^\infty_0(\Omega)$}
\]
and assume that $u\ge0$ a.e.\ in $\Omega$.
Then, either $u=0$ a.e.\ in $\Omega$ or $u>0$ a.e.\ in $\Omega$.
\end{proposition}

\begin{proof}
By \cite[Theorem A.1]{BF14},
$u>0$ in each connected component
where it is not identically zero. Then, we argue as
in the proof of \cite[Proposition 2.6]{BP}
and we prove a contrapositive statement: if $u\equiv0$ in a connected component $\Omega_0$ of $\Omega$, then, by assumption,
for all $\varphi\in C^\infty_0(\Omega_0)\setminus\{0\}$ such that $\varphi\ge0$,
\[
\int_{\Omega\setminus\Omega_0} 
	\int_{\Omega_0}\frac{u(x)\varphi(y)}{\abs{x-y}^{n+2s}}\,dx\,dy
-\frac{1}{2}\int_{\R^{N}}\int_{\mathbb R^N} \frac{(u(x)-u(y))(\varphi(x)-\varphi(y))}{\abs{x-y}^{n+2s}}\,dx\,dy \le0
\]	
which, by Fubini's theorem, implies
$u=0$ a.e.\ in $\Omega\setminus\Omega_0$, hence a.e.\ in $\Omega$.
\end{proof}



\begin{thebibliography}{99}
  
\bibitem{anello}
G. Anello, F. Faraci and A. Iannizzotto, On a problem of Huang concerning best constants in Sobolev embeddings, \emph{Ann.\ Mat.\ Pura Appl.}\ 194(3) (2015), 767--779.

\bibitem{BB}
M. van der Berg and D. Bucur, 
On the torsion function with Robin or Dirichlet boundary conditions, 
\emph{J. Funct.\ Anal.}\ 266(3) (2014), 1647--1666.

\bibitem{BFV}
M. Bonforte, A. Figalli and J.L. V\'azquez,
Sharp boundary behaviour of solutions to semilinear nonlocal elliptic equations, 
\emph{Calc.\ Var.\ Partial Differential Equations} 57(2) (2018).

\bibitem{B}
L. Brasco,
On principal frequencies and isoperimetric ratios in convex sets,
\emph{Ann.\ Fac.\ Sci.\ Toulouse Math.\ (6)} 29(4) (2020), 977--1005.

\bibitem{BC}
L. Brasco and E. Cinti, 
On fractional Hardy inequalities in convex sets, \emph{Discrete Contin.\ Dyn.\ Syst.}\ 38(8) (2018), 4019--4040.

\bibitem{BDF}
L. Brasco, G. De Philippis and G. Franzina,
Positive solutions to the sublinear Lane-Emden equation are isolated,
\emph{Comm.\ Partial Differential Equations} 46(10) (2021), 1940--1972.

\bibitem{BF14} 
L. Brasco and G. Franzina, 
Convexity properties of Dirichlet integrals and Picone-type inequalities, 
\emph{Kodai Math.\ J.} 37(3) (2014), 769--799.

\bibitem{BF19} 
L. Brasco and G. Franzina,
An overview on constrained critical points of Dirichlet integrals,
\emph{Rend.\ Semin.\ Mat.\ Univ.\ Politec.\ Torino} 78(2) (2020), 7--50.

\bibitem{BFR}
L. Brasco, G. Franzina and B. Ruffini,
Schr\"odinger operators with negative potentials and Lane-Emden densities,
\emph{J. Funct.\ Anal.}\ 274(6) (2018), 1825--1863.

\bibitem{BGCV}
L. Brasco, D. G\'omez-Castro and J.L. V\'azquez,
Characterisation of homogeneous fractional Sobolev spaces,
\emph{Calc.\ Var.\ Partial Differential Equations} 60(2) (2021).

\bibitem{BLP} 
L. Brasco, E. Lindgren and E. Parini, 
The fractional Cheeger problem, \emph{Interfaces Free Bound.}\ 16(3) (2014), 419--458.

\bibitem{BP} 
L. Brasco and E. Parini, 
The second eigenvalue of the fractional $p$-Laplacian, 
\emph{Adv.\ Calc.\ Var.}\ 9(4) (2016), 323--355.

\bibitem{CC}
Z.Q. Cheng and R. Song, 
Hardy inequality for censored stable processes,
\emph{Tohoku Math.\ J. (2)} 55(3) (2003), 439--450.

\bibitem{DL}
J. Deny and J.L. Lions,
Les éspaces du type de Beppo Levi, \emph{Ann.\ Inst.\ Fourier (Grenoble)} 
5 (1954), 305--370.

\bibitem{Knonloc}
A. Di Castro, T. Kuusi and G. Palatucci,
Local behavior of fractional $p$-minimizers,
\emph{Ann.\ Inst.\ H. Poincaré Anal.\ Non Linéaire} 33(5) (2016), 1279--1299.

\bibitem{D}
B. Dyda,
A fractional order Hardy inequality,
\emph{Illinois J. Math.}\ 48(2) (2004), 575--588.

\bibitem{DV0}
B. Dyda and A.V. V\"ah\"akangas,
A framework for fractional Hardy inequalities,
\emph{Ann.\ Acad.\ Sci.\ Fenn.\ Math.}\ 39 (2014), 675--689.

\bibitem{DV}
B. Dyda and A.V. V\"ah\"akangas,
Characterizations for fractional Hardy inequality, \emph{Adv.\ Calc.\ Var.}\ 8(2) (2015), 173--182.

\bibitem{F19} 
G. Franzina, 
Non-local torsion functions and embeddings,
\emph{Appl.\ Anal.}\ 98(10) (2019), 1811--1826.

\bibitem{HL} 
L. H\"ormander and J.L. Lions, 
Sur la complétion par rapport à une  intégrale de Dirichlet,
\emph{Math.\ Scand.}\ 4(2) (1956), 259--270.

\bibitem{MRS}
G. Molica Bisci, V.D. R\u{a}dulescu and R. Servadei, 
\emph{Variational Methods for Nonlocal Fractional Problems}, 
Encyclopedia Math.\ Appl.\ 162, Cambridge Univ.\ Press, Cambridge, 2016.

\bibitem{R2}
X. Ros-Oton,
Nonlocal elliptic equations in bounded domains: a survey, \emph{Publ.\ Mat.}\ 60(1) (2016), 3--26.

\bibitem{RS}
X. Ros-Oton and J. Serra,
The Dirichlet problem for the fractional Laplacian: Regularity up to the
boundary, \emph{J. Math.\ Pures Appl.}\ 101(3) (2014), 275--302.

\bibitem{Sk}
F. Sk,
Characterization of fractional Sobolev-Poincaré and (localized) Hardy inequalities, preprint (2022).

\bibitem{S} 
M. Struwe,
\emph{Variational Methods. Applications to Nonlinear Partial Differential Equations and Hamiltonian Systems}, Ergeb.\ Math.\ Grenzgeb.\ (3) 34, Springer, Berlin, 2008.

\bibitem{VV}
J.L. V\'azquez, 
Recent progress in the theory of nonlinear diffusion with fractional Laplacian operators, \emph{Discrete Contin.\ Dyn.\ Syst.\ Ser.\ S} 7(4) (2014), 857--885.

\end{thebibliography}
\end{document}